\documentclass[letterpaper,11pt,oneside,reqno]{amsart}

\usepackage{amsfonts,amsmath, amssymb,amsthm,amscd}
\usepackage{mathrsfs}
\usepackage{yfonts}
\usepackage{verbatim}
\usepackage{hyperref}
\usepackage{graphicx}
\usepackage[utf8]{inputenc}
\usepackage{latexsym}
\usepackage{upgreek}
\usepackage{relsize}
\usepackage{xcolor}
\usepackage{mathdots}
\usepackage{mathtools}
\usepackage{dsfont}
\usepackage[left=2.5cm, right=2.5cm, top=2.5cm, bottom=3cm]{geometry}
\usepackage{tikz}
\usepackage{lmodern}
\usepackage{stmaryrd}

\numberwithin{equation}{section}

\renewcommand{\emph}[1]{\textsf{\textit{#1}}}

\usepackage[width=.9\textwidth]{caption}

\let\oldtocsection=\tocsection
\let\oldtocsubsection=\tocsubsection
\renewcommand{\tocsection}[2]{\hspace{0em}\oldtocsection{#1}{#2}}
\renewcommand{\tocsubsection}[2]{\hspace{2em}\oldtocsubsection{#1}{#2}}

\begin{document}
\fontdimen8\textfont3=0.5pt  


\def \Ai {{\rm Ai}}
\def \Pf {{\rm Pf}}
\def \sgn {{\rm sgn}}
\def \SS {\mathcal{S}}
\newcommand{\EE}{\ensuremath{\mathbb{E}}}
\newcommand{\Var}{\mathrm{Var}}
\newcommand{\PP}{\ensuremath{\mathbb{P}}}
\newcommand{\R}{\ensuremath{\mathbb{R}}}
\newcommand{\C}{\ensuremath{\mathbb{C}}}
\newcommand{\Z}{\ensuremath{\mathbb{Z}}}
\newcommand{\N}{\ensuremath{\mathbb{N}}}
\newcommand{\Q}{\ensuremath{\mathbb{Q}}}
\newcommand{\T}{\ensuremath{\mathbb{T}}}
\newcommand{\I}{\ensuremath{\mathbf{i}}}
\newcommand{\Real}{\ensuremath{\mathfrak{Re}}}
\newcommand{\Imag}{\ensuremath{\mathfrak{Im}}}
\newcommand{\subs}{\ensuremath{\mathbf{Subs}}}
\newcommand{\dist}{\textrm{dist}}
\newcommand{\Res}[1]{\underset{{#1}}{\mathbf{Res}}}
\newcommand{\Resfrac}[1]{\mathbf{Res}_{{#1}}}
\newcommand{\Sub}[1]{\underset{{#1}}{\mathbf{Sub}}}
\newcommand{\la}{\lambda}
\newcommand{\ta}{\theta}
\newcommand{\labold}{\boldsymbol{\uplambda}}
\def\note#1{\textup{\textsf{\color{blue}(#1)}}}
\newcommand{\oldrho}{\rho}
\renewcommand{\rho}{\varrho}
\newcommand{\e}{\epsilon}
\newcommand{\eps}{\varepsilon}
\renewcommand{\geq}{\geqslant}
\renewcommand{\leq}{\leqslant}
\renewcommand{\ge}{\geqslant}
\renewcommand{\le}{\leqslant}

\newcommand{\qq}[1]{(q;q)_{#1}}
\newcommand{\bel}[1]{b_{#1}^{\mathrm{el}}}
\newcommand{\ve}{\mathcal{E}}
\newcommand{\Y}{\ensuremath{\mathbb{Y}}}
\newcommand{\Sym}{\ensuremath{\mathbf{Sym}}}

\newcommand{\kernel}{\mathsf{K}}
\newcommand{\fkernel}{\mathsf{f}}
\newcommand{\gkernel}{\mathsf{g}}
\newcommand{\hkernel}{\mathsf{h}}
\newcommand{\scaling}[1]{ \mathfrak{s}(#1)}
\def \Pf {{\rm Pf}}

\newcommand{\p}{\mathsf{p}}
\newcommand{\q}{\mathsf{q}}
\newcommand{\ratealpha}{\upalpha}
\newcommand{\rategamma}{\upgamma}
\newcommand{\ratebeta}{\upbeta}
\newcommand{\ratedelta}{\updelta}
\newcommand{\gap}{\ensuremath{\mathrm{gap}}}
\newcommand{\Weyl}[1]{\mathbb{W}^{#1}}

\newcommand{\rightarrowplus}{\overset{+}{\longrightarrow}}
\newcommand{\leftarrowplus}{\overset{+}{\longleftarrow}}
\newcommand{\rightarrowminus}{\overset{-}{\longrightarrow}}
\newcommand{\leftarrowminus}{\overset{-}{\longleftarrow}}


\newcommand{\Leta}{\mathscr{L}^{\mathbb Z_{>0}}}
\newcommand{\Letafullspace}{\mathscr{L}^{\mathbb Z}}
\newcommand{\Letasegment}{\mathscr{L}^{\llbracket 1,\ell-1\rrbracket}}
\newcommand{\Lpart}{\mathscr{D}}
\newcommand{\Lzero}{\mathscr{L}^{\circ}}
\newcommand{\Lzerozero}{\mathscr{L}^{\circ\circ}}


\usetikzlibrary{patterns}
\usetikzlibrary{shapes.multipart}
\usetikzlibrary{arrows}

\tikzstyle{axis}=[->, >=stealth', thick, gray]
\tikzstyle{grille}=[dotted, gray]
\tikzstyle{path}=[->, >=stealth', thick]


\newtheorem{theorem}{Theorem}[section]
\newtheorem{conjecture}[theorem]{Conjecture}
\newtheorem{lemma}[theorem]{Lemma}
\newtheorem{proposition}[theorem]{Proposition}
\newtheorem{corollary}[theorem]{Corollary}

\newtheorem{theoremintro}{Theorem}
\renewcommand*{\thetheoremintro}{\Alph{theoremintro}}

\theoremstyle{definition}
\newtheorem{remark}[theorem]{Remark}

\theoremstyle{definition}
\newtheorem{example}[theorem]{Example}

\theoremstyle{definition}
\newtheorem{definition}[theorem]{Definition}

\theoremstyle{definition}
\newtheorem{definitions}[theorem]{Definitions}


\title{\large Markov duality and Bethe ansatz formula for half-line open ASEP}

\author[G. Barraquand]{Guillaume Barraquand}
\address{G. Barraquand,
	Laboratoire de Physique de l'Ecole Normale Supérieure, 
	ENS, Université PSL, CNRS, Sorbonne Université, Université Paris Cité, F-75005 Paris, France}
\email{barraquand@math.cnrs.fr}
\author[I. Corwin]{Ivan Corwin}
\address{I. Corwin, Columbia University,
	Department of Mathematics,
	2990 Broadway,
	New York, NY 10027, USA.}
\email{ivan.corwin@gmail.com}

\begin{abstract}
Using a Markov duality satisfied by ASEP on the integer line, we deduce similar dualities for half-line open ASEP and open ASEP on a segment. This leads to closed systems of ODEs characterizing observables of the models. In the half-line case, we solve the system of ODEs using Bethe ansatz and prove an integral formula for $q$-moments of the integrated current at $n$ distinct spatial locations. We then use this formula to confirm predictions for the moments of the multiplicative noise stochastic heat equation on $\mathbb R_{>0}$ with Robin type boundary condition and we obtain new formulas in the case of a Dirichlet boundary condition.   
\end{abstract}

\maketitle


\section{Introduction}
The asymmetric simple exclusion process (ASEP) on the one-dimensional lattice is an exactly solvable model, in the sense that the distribution of the system -- and other observables -- can be computed exactly using coordinate Bethe ansatz. This was first discovered in the totally asymmetric case  \cite{schutz1997exact} (see also \cite{johansson2000shape} for exact formulas of a different flavour). The partially asymmetric case  was then treated in \cite{tracy2008integral, tracy2008fredholm, tracy2011erratum} and led to important progress on the asymptotic behaviour of ASEP \cite{tracy2009asymptotics} and the KPZ equation \cite{amir2011probability, sasamoto2010exact}. Quite often, exactly solvable probabilistic models admit half-space variants that are also exactly-solvable  and can be analyzed asymptotically \cite{baik2001algebraic, baik2001asymptotics,sasamoto2004fluctuations,  gueudre2012directed, o2014geometric,baik2018pfaffian,  betea2018free, barraquand2018stochastic, barraquand2018half, Krajenbrink2020, barraquand2022random, imamura2022solvable, he2022shift}. More generally, in the context of quantum integrable systems -- where the coordinate Bethe ansatz method originally comes from \cite{bethe1931theorie} --  there exists a huge literature studying how to impose boundary conditions that preserve integrability. For the half-line open ASEP, however, exact computations of the distribution of the integrated  current and its asymptotic analysis have remained mostly out of reach. 

There are two notable exceptions. In the special case of the ASEP on $\Z_{>0}$ with closed boundary conditions (no particles injected or ejected at the boundary), \cite{tracy2013bose}  computed the transition probabilities of the system via coordinate Bethe ansatz. An application to open ASEP was discussed in \cite{tracy2013asymmetric}, but the resulting formulas are fully explicit only in the special cases of the totally asymmetric (TASEP) and symmetric (SSEP) exclusion processes. For special values of the boundary parameters (imposing a density $\rho=1/2$, according to the notations used below), one exact formula was also obtained via the framework of  half-space Macdonald processes in \cite{barraquand2018stochastic}, and allows to analyze the system asymptotically. However, that formula characterizes only the distribution of the integrated  current at the origin, and for a special value of boundary parameters. After the posting of this paper on arXiv, \cite{he2023boundary} extended the results of \cite{barraquand2018stochastic}, using identities from \cite{imamura2021skew}, allowing two generic boundary parameter, though still restricting to the study of the integrated current at the origin. 

An immediate difficulty that arises with  open systems to apply coordinate Bethe ansatz is that the number of particles is not conserved. This issue can be resolved using an idea that was already very useful for the ASEP on $\Z$. Using coordinate Bethe ansatz, \cite{tracy2008integral} obtained integral formulas for transition probabilities $P_t(\vec x,\vec y)$ for ASEP on $\Z$, where $\vec x$ and $\vec y$ denote the positions of $N$ particles locations. They further showed that certain observables, for instance the distribution of a tagged particle, can be written as a formula where one may let the number $N$ of particles go to infinity, and that such observables can be written as Fredholm determinants \cite{tracy2008fredholm}. Analyzing such Fredholm determinants asymptotically is highly non-trivial and required clever manipulations to transform Fredholm determinants \cite{tracy2009asymptotics}. Another approach to analyze ASEP on $\Z$ was proposed in \cite{imamura2011current, borodin2012duality}, using a Markov duality, originally discovered in \cite{schutz1997duality}. We review this notion in Section \ref{sec:duality}. This approach through duality allows, in general, to study particle systems on a complicated state space (e.g. a system with possibly infinitely many particles) by reducing the problem to diagonalizing a dual Markov process on a more convenient state space (e.g. a system involving finitely many particles). In the case of the full-line ASEP, the dual system is the ASEP with a fixed finite number of particles and is thus Bethe ansatz solvable. 

In this paper, we adapt this approach to the half-line open ASEP. In that case, the number of particles in the system is a priori random, and it is useful to reduce to a dual particle system with a fixed number of particles. We do not solve the most general case but impose a condition on the boundary injection/ejection rate that was already used in \cite{liggett1975ergodic}.
 We show that this condition  allows to deduce a Markov duality for half-line open ASEP from the full-space duality. Despite the restriction on boundary parameters, we can still study the system with an arbitrary boundary density parameter. 
 Our main result, Theorem \ref{theo:momentformula}, is an exact integral formula for the $q$-moments of the integrated current at $n$ distinct spatial locations (the parameter $q$ denotes the asymmetry, see Definition \ref{def:halflineASEP} below). It is obtained by explicitly solving a system of ODEs coming from the duality (Proposition \ref{prop:freeevolutionu}) which characterizes these $q$-moments. Our solution takes the form of a sum over partitions of integrals having a structure similar with solutions of the half-space delta Bose gas from \cite{borodin2016directed}. 

 Markov dualities for open ASEP on a half-line or a segment have been investigated in the literature  \cite{ohkubo2017dualities, kuan2021algebraic, kuan2019stochastic, kuan2021short,  kuan2022two, schutz2022reverse} under various assumptions on boundary rates and using different duality functionals. We review these works and compare them to our results in Remark \ref{rem:dualityhalfline} (in the case of half-line open ASEP) and Remark \ref{rem:dualityopenasep} (in the case of open ASEP on a segment). 
  Our duality  (Theorem \ref{theo:dualityhalf-space}) is a duality between half-line open ASEP generator and a sub-Markovian generator corresponding to  half-line ASEP with a constant number $n$ of particles elastically reflected at the boundary. As a side result, we also explain in Section \ref{sec:segment} that in the case of open ASEP on a segment $\lbrace 1, \dots, \ell\rbrace$, our approach also yields a duality (Theorem \ref{theo:dualitysegment}). This allows as well to deduce a closed system of ODEs characterizing the $q$-moments of the integrated current, but so far we do not know how to explicitly solve it. 

Since ASEP converges in the weakly asymmetry limit to the Kardar-Parisi-Zhang (KPZ) equation \cite{bertini1997stochastic, corwin2016open, parekh2017kpz, gonccalves2020derivation, yang2020kardar}, our formula also sheds light on the half-line KPZ equation. In particular, we prove moment formulas for the KPZ equation with Neumann type boundary (Theorem \ref{theo:momentsKPZ}). Although it seems to us that Theorem \ref{theo:momentsKPZ} has not been proved in the literature, the result was anticipated in \cite{borodin2016directed} and could be proved alternatively using other discrete models (see Remark \ref{rem:otherapproach}). However,  we also obtain new results that are related to the half-line KPZ equation with Dirichlet boundary condition (in the sense that the exponential of the solution satisfies a Dirichlet boundary condition). These results are used in the forthcoming paper \cite{barraquand2022weakly} to prove the convergence of half-line ASEP with high boundary density to the half-line KPZ equation with Dirichlet boundary condition.

\subsection*{Acknowledgments} We thank Oriane Blondel, Chiara Franceschini, Jeffrey Kuan,  Yier Lin, Tomohiro Sasamoto and Marielle Simon for useful conversations related to the results in this paper. I.C. was partially supported by the NSF through grants DMS:1937254, DMS:1811143, DMS:1664650, as well as through a Packard Fellowship in Science and Engineering, a Simons Fellowship, a Miller Visiting Professorship from the Miller Institute for Basic Research in Science, and a W.M. Keck Foundation Science and Engineering Grant. G.B. was partially supported by NSF through grant DMS:1664650 and by the ANR grants ANR-19-CE40-0012 and ANR-21-CE40-0019. Both G.B. and I.C. also wish to acknowledge the NSF grant DMS:1928930 which supported their participation in a fall 2021 semester program at MSRI in Berkeley, California. 

\section{An exact formula for the integrated currents in half-line open ASEP} 
\label{sec:halfline}
We start with some definitions and notations. 
\begin{definition}
	The \emph{half-line (open) ASEP} is a continuous time Markov process on the state-space $\lbrace 0,1\rbrace^{\Z_{>0}}$. It describes the stochastic evolution of a system of particles on $\Z_{>0}$ satisfying the exclusion rule -- each site is occupied by at most one particle. The particle configurations can be described by occupation variables $\big\lbrace \eta =  \big(\eta_x\big)_{x\in\Z_{>0}}\in \lbrace 0,1\rbrace^{\Z_{>0}} \big\rbrace$ where $\eta_x=0$ when the site $x$ is empty, and $\eta_x=1$ when it is occupied by a particle. The state $\eta(t)$ at time $t$ evolves according to the following dynamics, depicted in Fig. \ref{fig:generalASEP}, depending on nonnegative real parameters $\p, \q, \ratealpha$ and $\rategamma$: for any $x\in \Z_{>0}$, a particle jumps from site $x$ to $x+1$ at exponential rate
	\begin{equation*}\p\ \eta_x(1-\eta_{x+1}) \in \lbrace 0,\p\rbrace, 
	\end{equation*} and from site $x+1$ to $x$ at exponential rate 
\begin{equation*}\q\ \eta_{x+1}(1-\eta_x)\in\lbrace 0,\q\rbrace.
\end{equation*} 
Further,  a particle is created or annihilated at the site $1$ at exponential rates
	\begin{equation*} \ratealpha\ (1-\eta_1) \ \  \text{and}\ \  \rategamma\ \eta_1.
	\end{equation*}
	All these events are independent. We will restrict our attention to the initial conditions with finitely many particles in the system, such as the empty initial condition. Following \cite{schutz1997duality}, we define the integrated current at site $x$ by
	\begin{equation} N_x(\eta) = \sum_{i=x}^{+\infty} \eta_i.
		\label{eq:defNx}
		\end{equation} 
	We denote by $\Leta$ the generator associated to this Markov process, that is the operator acting on local  functions $f: \lbrace 0,1\rbrace^{\Z_{>0}} \to \R$ by 
	\begin{multline}
	\Leta f(\eta)  =  \ratealpha (1-\eta_1) \left( f(\eta^+)-f(\eta)\right) + \rategamma \eta_1 \left( f(\eta^-)-f(\eta)\right) \\ 
	+ \sum_{x=1}^{\infty} \left(  \p\eta_x(1-\eta_{x+1}) +  \q\eta_{x+1}(1-\eta_{x}) \right) \left( f(\eta^{x,x+1})-f(\eta)\right), 
	\label{eq:defLeta}
	\end{multline}
	where $\eta^{x,x+1}$ is the configuration obtained from $\eta$ by exchanging the values of $\eta_x$ and $\eta_{x+1}$, and $\eta^+$ (resp. $\eta^-$) is obtained from $\eta$ by setting $\eta_1=1$ (resp. $\eta_1=0$). 
	\label{def:halflineASEP}
\end{definition}
\begin{figure}
			\vspace{-0.3cm}
	\begin{tikzpicture}[scale=0.7]
	\draw[thick] (-1, 0) circle(1);
	\draw (-1,0) node{\footnotesize reservoir};
	\draw[thick] (0, 0) -- (12.5, 0);
	\foreach \x in {1, ..., 12} {
		\draw[gray] (\x, 0.15) -- (\x, -0.15) node[anchor=north]{\footnotesize $\x$};
	}
	
	\fill[thick] (3, 0) circle(0.2);
	\fill[thick] (6, 0) circle(0.2);
	\fill[thick] (7, 0) circle(0.2);
	\fill[thick] (10, 0) circle(0.2);
	\draw[thick, -stealth] (3, 0.3)  to[bend left] node[midway, above]{$\p$} (4, 0.3);
	\draw[thick, -stealth] (6, 0.3)  to[bend right] node[midway, above]{$\q$} (5, 0.3);
	\draw[thick, -stealth] (7, 0.3) to[bend left] node[midway, above]{$\p$} (8, 0.3);
	\draw[thick, -stealth] (10, 0.3) to[bend left] node[midway, above]{$\p$} (11, 0.3);
	\draw[thick, -stealth] (10, 0.3) to[bend right] node[midway, above]{$\q$} (9, 0.3);
	\draw[thick, -stealth] (-0.1, 0.5) to[bend left] node[midway, above]{$\ratealpha$} (0.9, 0.4);
	\draw[thick, stealth-] (0, -0.5) to[bend right] node[midway, below]{$\rategamma$} (0.9, -0.4);
	\end{tikzpicture}
			\vspace{-0.3cm}
	\caption{Jump rates in the half-line ASEP.}
	\label{fig:generalASEP}
\end{figure}
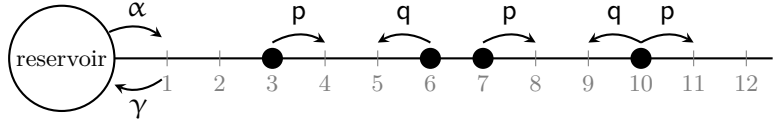

We will denote the ratio of jump rates  by $q:= \q/\p$ and assume that $q\in (0,1)$. We define the functions
\begin{equation}  
	Q_x(\eta) := q^{N_x(\eta)}.
	\label{eq:defQx}
\end{equation} 
Finally, most of our results hold under a condition on the boundary parameters, 
\begin{equation}
\frac{\ratealpha}{\p}+\frac{\rategamma}{\q} = 1.
\label{eq:Liggettcondition}
\end{equation} 
This condition comes from the work \cite[see page 240]{liggett1975ergodic}. Introducing the density parameter  $\rho:=\ratealpha/\p$, the condition \eqref{eq:Liggettcondition} ensures that  the injection rate is $\p \rho$ and the ejection rate is $\q(1-\rho)$, as if the injection and ejection rates resulted from a fictitious site occupied with probability $\rho$ and the particles were traveling between this fictitious site and the site $1$ with the same rates $\p, \q$ as in the bulk. We will make this interpretation more precise in Section \ref{sec:dualityhalfline}.

\subsection{Background on Markov duality}
\label{sec:duality}

We say that two Markov processes $X(t)$ and $Y(t)$, defined respectively on state-spaces $\mathbb X$ and $\mathbb Y$, are dual with respect to a function $H:\mathbb X\times \mathbb Y\to \R$, if for all $x\in\mathbb X, y\in\mathbb Y$, and all $t\geq 0$, 
\begin{equation} \mathbb E^x\left[H(X(t),y )  \right] =  \mathbb E_y\left[H(x, Y(t) )   \right],
\label{eq:defduality}
\end{equation}
where in the L.H.S., $\mathbb E^x$ denotes the expectation with respect to the Markov process $X(t)$ started from $X(0)=x$, and in the R.H.S., $\mathbb E_y$ denotes the expectation with respect to the Markov process $Y(t)$ started from $Y(0)=y$. When $\mathbb X=\mathbb Y$ and $X_t$ and $Y_t$ have the same distribution, we say that the Markov process $X_t$ satisfies a self-duality with respect to $H$. 
Going back to the general case, taking derivatives in \eqref{eq:defduality} with respect to $t$, and assuming the existence of generators of these processes, Markov duality implies that 
\begin{equation}
\mathscr L^X H(x,y)=\mathscr L^Y H(x,y),
\label{eq:dualitygenerator}
\end{equation} 
where $\mathscr L^X$ (resp. $\mathscr L^Y$) acts on functions from $\mathbb X$ to $\R$ (resp. from $\mathbb Y$ to $\R$) and denotes the generator of the Markov process $X(t)$ (resp. $Y(t)$). When $H$ is uniformly bounded and belongs to the domain of $L^X$ for any $y$ (resp. to the domain of $L^Y$ for any $x$), then  \eqref{eq:dualitygenerator} also implies \eqref{eq:defduality} (see \cite[Prop. 1.2]{jansen2014}).

The full-space ASEP is the Markov process on $\lbrace 0,1\rbrace^{\mathbb Z}$ with generator $\Letafullspace$, acting on local functions $f:\lbrace 0,1\rbrace^{\Z}\to\R$ by 
\begin{equation}
\Letafullspace f(\eta)  =  
\sum_{x\in \mathbb Z} \left(  \p\eta_x(1-\eta_{x+1}) +  \q\eta_{x+1}(1-\eta_{x}) \right) \left( f(\eta^{x,x+1})-f(\eta)\right). 
\label{eq:defgeneratorfullspace}
\end{equation}
In order to state ASEP's duality, it is convenient to introduce, for  $y\in \Z$, the sets  
\begin{equation*}  \Weyl{n}_{\geqslant y} :=\left\lbrace \vec x\in \Z^n : y\leqslant x_1<\dots < x_n \right\rbrace, \;\;\;\;\Weyl{n} :=\left\lbrace \vec x\in \Z^n :  x_1<\dots < x_n \right\rbrace.
\end{equation*}
We introduce another operator $\Lpart^{(n)}$ acting on functions $f:\Weyl{n}\to \R$ by  
\begin{equation}
	\Lpart^{(n)} f(\vec x)  = \sum_{\underset{x_i-x_{i-1}>1}{1\leqslant i\leqslant n} } \p\left( f(\vec x_i^-) -f(\vec x)\right) + \sum_{   \underset{x_{i+1}-x_{i}>1}{1\leqslant i\leqslant n}} \q\left( f(\vec x_i^+) -f(\vec x)\right),
	\label{eq:defLpart}
\end{equation}
where for $\vec x\in \Weyl{n}$, $\vec x_i^{\pm}:=(x_1, \dots, x_i\pm 1, \dots, x_n)$, and we use the notational convention that $x_0=-\infty$ and $x_{n+1}=+\infty$.
The operator $\Lpart^{(n)} $ is also the generator of a continuous Markov process on the state space $\Weyl{n}$: the full-space ASEP with $n$ particles, where the state of the particle system is now described by particle positions  $x_i$ rather than occupation variables $\eta_i$ as previously.
Note that in \eqref{eq:defgeneratorfullspace}, $\p$ and $\q$ are the jump rates associated to right and left, while in \eqref{eq:defLpart}, we have chosen the opposite  orientation. We will use the following Markov duality.  
\begin{proposition}[{\cite[Theorem 4.2, Proposition 4.7]{borodin2012duality}}] Fix $n\geqslant 1$. 
The full-space ASEP with right and left jump rates $\p$ and $\q$ (i.e. the Markov process with generator \eqref{eq:defgeneratorfullspace}) and  the $n$-particle ASEP with right and left jump rates $\q$ and $\p$ (i.e. the Markov process with generator \eqref{eq:defLpart}) are dual with respect to 
\begin{equation}
H(\eta,\vec x) := \prod_{i=1}^n Q_{x_i}(\eta),
\label{eq:defH}
\end{equation}
where we have extended the definitions of $Q_x(\eta)$ and $N_x(\eta)$ in \eqref{eq:defQx} and \eqref{eq:defNx} to negative values of $x$ and $\eta\in \lbrace 0,1\rbrace^{\Z}$. 
In particular, for any $\eta\in \lbrace 0,1\rbrace^{\Z}$ and $\vec x\in \Weyl{n}$, 
\begin{equation*}
\Letafullspace H(\eta, \vec x) = \Lpart^{(n)} H(\eta, \vec x). 
\label{eq:dualityfullspace}
\end{equation*}
\label{prop:dualityfull-space}
\end{proposition}
Since $\Lpart^{(n)}$ is the generator of full-space ASEP, Proposition \ref{prop:dualityfull-space} can also be rephrased as a self-duality, though for our purposes, it is convenient to describe the two processes differently. A Markov (self-)duality for ASEP was first obtained by \cite{schutz1997duality} with respect to a different function, the function $\widetilde H(\eta, \vec x):= \prod_{i=1}^n\widetilde Q_{x_i}(\eta)$ where $\widetilde Q_{x}(\eta)=\eta_xq^{N_{x+1}(\eta)}$. It was inspired by the analysis of the XXZ spin chain and proved  by writing ASEP generator as a tensor product of Pauli matrices and expressing symmetries in terms of the $U_q(\mathfrak{sl}_2)$ algebra. Another proof was provided in \cite{borodin2012duality} by writing out explicitly $\Letafullspace \widetilde H$ and $\Lpart^{(n)} \widetilde H$ and comparing the two expressions. Proposition \ref{prop:dualityfull-space} states a Markov duality with respect to the function $H$, and it is proved in \cite{borodin2012duality}, also using an explicit computation of how generators act on $H$. 
We refer to \cite[Section 2.6.1]{kuan2019stochastic} for a discussion of several Markov dualities satisfied by  ASEP. 

\subsection{Duality for half-line ASEP under Liggett's condition}
\label{sec:dualityhalfline}
In this Section we will use the full-space ASEP duality from Proposition \ref{prop:dualityfull-space} to establish a duality for the half-line open ASEP under Liggett's condition \ref{eq:Liggettcondition}. 
Although we will see that it can be stated as a Markov duality, it is convenient to first introduce the following weaker notion. 
\begin{definition}
Let $\mathbb X$ and $\mathbb Y$ be two state spaces. We say that the operator $\mathscr L^X$, acting on functions $\mathbb X\to \mathbb R$ and the operator $\mathscr D^{Y}$, acting on functions $\mathbb Y\to \mathbb R$, are dual with respect to the function $H:\mathbb X\times\mathbb Y\to \mathbb R$ if for all $x\in \mathbb X$ and $y\in \mathbb Y$, 
$$ \mathscr L^X H(x, y) = \mathscr D^Y H(x,y).$$
\label{def:dualityoperators}
\end{definition}
In the following, the operator $\mathscr L^X$ will be the generator of a Markov process,  and  the operator $\mathscr D^Y$ will be of the form $\mathscr D_{Y}=\mathscr L^Y+ V$ where $\mathscr L^Y$ is the generator of a Markov process on the state space $\mathbb Y$ and $V:\mathbb Y\to \mathbb R$ is a bounded function. This notion of duality was considered in e.g. \cite[Chap 4, Sec. 4]{ethier2009markov}, and can be rephrased, via the Feynman-Kac formula,  as an equality between expectations (See Corollary \ref{cor:halfspace} and Corollary \ref{cor:segment} below).

Define an operator  $\Lpart^{(n, \rho)}$ acting on functions $f:\Weyl{n}_{\geqslant 1}\to \R$ by  
\begin{multline}
	\Lpart^{(n, \rho)} f(\vec x)  = \sum_{\underset{x_i-x_{i-1}>1}{2\leqslant i\leqslant n} } \p\left( f(\vec x_i^-) -f(\vec x)\right) + \sum_{   \underset{x_{i+1}-x_{i}>1}{1\leqslant i\leqslant n}} \q\left( f(\vec x_i^+) -f(\vec x)\right)\\ 
	 +\mathds{1}_{x_1>1}\p\left( f(\vec x_1^-) -f(\vec x)\right) -\mathds{1}_{x_1=1} (\p-\q)\rho f(\vec x),
	\label{eq:defLpartboundary}
\end{multline}
with the convention that $x_{n+1}=+\infty$. This operator is a sub-Markov generator, which describes the time evolution of the transition probability of a system of $n$ particles on $\Z_{\geqslant 1}$ performing continuous time simple random walks with jump rates $\p$ and $\q$, under the exclusion constraint, with killing at the boundary. More precisely, when the first particle is on site $1$, it cannot jump left and the process is killed with rate $(\p-\q)\rho$ (we refer to \cite{fitzsimmons1999kac} for more details on Markov processes killed at a state-dependent rate). We will call $\Lpart^{(n, \rho)}$ the (sub-Markov) generator of the \emph{killed $n$-particle half-line ASEP}. We now state a first duality result in the form that is the most convenient for our later purposes. We then show in Section \ref{sec:markovdualityhalfspace} that it implies a Markov duality (Corollary \ref{cor:markovdualityhalfspace}). We also provide another probabilistic interpretation in Corollary \ref{cor:halfspace} and further comment on the killing mechanism in Remark \ref{rem:killing}. 
\begin{theorem} Fix $n\geqslant 1$ and assume that \eqref{eq:Liggettcondition} holds. 
	The generator \eqref{eq:defLeta} of the half-line open  ASEP with right and left jump rates $\p$ and $\q$ and the operator \eqref{eq:defLpartboundary}, generator of the killed $n$-particle half-line ASEP with right and left jump rates $\q$ and $\p$, are dual with respect to 
	\begin{equation*}
		H(\eta,\vec x) := \prod_{i=1}^n Q_{x_i}(\eta).
	\end{equation*}
	Equivalently, for any $\eta\in \lbrace 0,1\rbrace^{\Z\geqslant 1}$ and $\vec x\in \Weyl{n}_{\geqslant 1}$, 
	\begin{equation*}
		\Leta H(\eta, \vec x) = \Lpart^{(n, \rho)} H(\eta, \vec x). 
		\label{eq:dualityhalfspace}
	\end{equation*}
	\label{theo:dualityhalf-space}
\end{theorem}
The special case of Theorem \ref{theo:dualityhalf-space} for $n=1$ was discovered in \cite{corwin2016open} and can be interpreted as a microscopic Hopf-Cole transform. It was used in \cite{corwin2016open}  to prove the convergence of half-line ASEP (as well as open ASEP on the segment) to the KPZ equation with appropriate boundary conditions. 
\begin{proof} 
Liggett's condition \eqref{eq:Liggettcondition} allows to rewrite the generator $\Leta$ by making use of some fictitious occupation variable $\eta_0$ at site $0$, where the occupation variable $\eta_0$ should be thought of as a Bernoulli random variable with parameter $\rho$  (recall that $\rho=\ratealpha/\p$), independent from all other random variables.

	More precisely, let us first define, for any local function $f:\lbrace 0,1\rbrace^{\Z_{\geq 0}}$, 
	\begin{equation}
		\Lzero f(\eta)  =  \sum_{x=0}^{\infty} \left(  \p\eta_x(1-\eta_{x+1}) +  \q\eta_{x+1}(1-\eta_{x}) \right) \left( f(\eta^{x,x+1})-f(\eta)\right), 
		\label{eq:defLeta0}
	\end{equation}
	that is the generator of half-line ASEP on $\mathbb Z_{\geq 0}$ (instead of $\mathbb Z_{>0}$ as above), with no injection nor ejection of particles at site $0$. Let $E$ denote the expectation with respect to the probability distribution on $\eta_0$ such that $\eta_0=1$ with probability $\rho$, and $0$ with probability $1-\rho$. The variable $\eta_0$ arises in \eqref{eq:defLeta0} only in the term $x=0$, and under \eqref{eq:Liggettcondition}, we have $E[\p\eta_0]=\ratealpha$ and $E[\q(1-\eta_0)]=\rategamma$. Hence for any local function $f:\lbrace 0,1\rbrace^{\mathbb Z_{>0}}\to \R$,  
	\begin{equation}
		\Leta f(\eta)=E\left[\Lzero f (\eta)\right]. 
		\label{eq:Liggetttrick}
	\end{equation}
	
	Now, we may use the duality for full-space ASEP. First, we remark that  the function ${H}(\eta,\vec x)$ was  defined for $\eta\in \lbrace 0,1\rbrace^{\mathbb Z}$ in \eqref{eq:defH}, but when $x_1\geq 1$, it depends only on the projection of $\eta$ to the positive occupation variables, that is, it depends only on $(\eta_1,\eta_2,\dots)\in \lbrace 0,1 \rbrace^{\mathbb Z_{>0}}$. Moreover, for $x_1\geqslant 1$,  when computing $\Letafullspace {H}(\eta,\vec x)$, the summation over all $x\in \mathbb Z$ in \eqref{eq:defgeneratorfullspace} can be restricted to a summation over $x\geq 0$ (because for $x< 0$ and $x_1\geq 1$, ${H}(\eta^{x,x+1},\vec x)={H}(\eta,\vec x)$). 
	This implies that for any $\vec x\in \Weyl{n}_{\geq 1}$ and $\eta\in \lbrace 0,1\rbrace^{\mathbb Z_{>0}}$, we have  $\Letafullspace {H}(\eta,\vec x)=\Lzero {H}(\eta,\vec x)$. Thus, using \eqref{eq:Liggetttrick}, for all $\vec x\in \Weyl{n}_{\geq 1}$, 
	\begin{equation}
		\Leta H(\eta, \vec x)=E\left[\Lzero H(\eta; \vec x)\right]= E\left[\Letafullspace H(\eta; \vec x)\right].
		\label{eq:Liggetttrickapplied}
	\end{equation}
	Using now Proposition \ref{prop:dualityfull-space}, 
	\begin{equation}
		\Leta H(\eta, \vec x)= E\left[\Lpart^{(n)} H(\eta; \vec x)\right],
		\label{eq:Liggetttrickduality}
	\end{equation}
	where the expectation $E$ simply means that we average over $\eta_0$, which arises in $Q_0(\eta)=q^{\eta_0}Q_1(\eta)$. The right-hand-side of \eqref{eq:Liggetttrickduality} can be seen as a polynomial in $Q_0(\eta),Q_1(\eta),Q_2(\eta),\dots$ which is linear in $Q_0(\eta)$. Since only $Q_0(\eta)$ depends on $\eta_0$, the expectation $E$ can be computed by replacing each occurrence of $Q_0(\eta)$ by $(\rho q+1-\rho)Q_1(\eta)$. There is actually at most one term involving $Q_0(\eta)$ in $\Lpart^{(n)} H(\eta; \vec x)$, arising in the term $i=1$ in \eqref{eq:defLpart} when $x_i=1$,  and since 
\begin{equation*} \p\Big((\rho q+1-\rho)-1\Big) = \rho(\q-\p),
\end{equation*}
	we obtain that for any $\eta\in \lbrace 0,1\rbrace^{\Z\geqslant 1}$ and $\vec x\in \Weyl{n}_{\geqslant 1}$, 
	\begin{equation*}\Leta H(\eta, \vec x)= \Lpart^{(n, \rho)} H(\eta; \vec x).
	\end{equation*}
	\end{proof} 
	\begin{remark}
		Markov dualities for half-line ASEP or open ASEP on a segment have been investigated in earlier references 
		 under various assumptions on boundary rates, involving different observables than the function $H(\eta, \vec x)$. In particular, in the half-line case,  a duality between half-line open ASEP with parameters $\alpha=1, \gamma=0$ and the same process with parameters $\rategamma=1, \ratealpha=0$ was found in \cite{kuan2019stochastic}, with respect to Schütz's functional. In \cite{kuan2022two} another  self-duality result was proved for half-line open ASEP with rates $ \p\in (0,1), \q=1, \ratealpha\in (0,1)$ and $\rategamma=1$, but the duality functional is more complicated (the expression is given in \cite[Theorem 1.3]{kuan2022two}).  
		\label{rem:dualityhalfline}
	\end{remark}
\begin{remark}
	The fact that the full-line ASEP also satisfies a duality with respect to the function  $\widetilde H(\eta, \vec x)$  might be used to show some analogue of Theorem \ref{theo:dualityhalf-space} involving the functional  $\widetilde H(\eta(t), \vec x)$. However, the boundary condition in that case would take a significantly more complicated form, so we chose not to pursue further this direction. More generally, Markov dualities have been proved in the last few years for a number of other Markov processes  \cite{corwin2015stochastic, carinci2016generalized, carinci2016asymmetric, kuan2018multi, carinci2019orthogonal, carinci2019consistent, kuan2021short,  chen2020integrable, kuan2020interacting, carinci2021q, franceschini2022orthogonal}. It would be interesting if our method using fictitious sites applies to some of these other systems as well. 
\end{remark}
\begin{remark}
	If we do not assume condition \eqref{eq:Liggettcondition}, it does not seem that one can prove a duality with respect to the function $H(\eta, \vec x)$. More precisely, it can be shown that for all $\vec x\in \Weyl{n}_{\geqslant 2}$ and $\eta\in \lbrace 0,1\rbrace^{\mathbb Z_{>0}}$,
	\begin{equation*}
\Leta {H}(\eta,\vec x)  = \Lpart^{(n)}H(\eta,\vec x),
	\end{equation*}
	while if $x_1=1$, 
	\begin{equation*} \Leta {H}(\eta,\vec x)  =  (\ratealpha q+\rategamma)H(\eta, 2,x_2, \dots, x_n) -(\ratealpha+\rategamma)H(\eta, 1, x_2, \dots, x_n)  + \Lpart^{(n-1)}H(\eta,1, x_2, \dots, x_n)
	\end{equation*}
	where $\Lpart^{(n-1)}$ is understood as acting on functions of $(x_2< \dots< x_n)$. Hence, $\Leta {H}(\eta,\vec x)$  involves the values of $H(\eta, \vec y)$ for some $\vec y\not\in \Weyl{n}_{\geq 1}$ (when $x_1=1$ and $x_2=2$). 
	\label{rem:withoutLiggett}
\end{remark}

\subsection{Markov duality}
\label{sec:markovdualityhalfspace}
Theorem \ref{theo:dualityhalf-space} is a duality of operators between the generator $\Leta$ and the sub-Markov generator $\Lpart^{(n,\rho)}$. In this section, we show that it implies a Markov duality in the sense of \eqref{eq:defduality}.  

A sub-Markov process can often be seen as a Markov process on a state space augmented by a cemetery state which we will denote by the symbol $\partial$ (following the notations in \cite{fitzsimmons1999kac}). Consider the state space $\mathsf W^{(n)}:=\Weyl{n}_{\geq 1}\cup \lbrace \partial \rbrace$ and the continuous-time Markov process $\mathsf x(t)\in \mathsf W^{(n)}$ such that when $\mathsf x(t)=\vec x\in \Weyl{n}_{\geq 1}$, the dynamics are the same as in the $n$-particle ASEP on $\mathbb Z_{>0}$ with closed boundary conditions and jump rates $\p$ to the left and $\q$ to the right, except that when $x_1=1$, the process jumps to the cemetery state $\partial$ at rate $\rho(\p-\q)$. Formally, this is the Markov process on $\mathsf W^{(n)}$ with generator $\Lpart^{n,\partial}$ defined on functions $f:\mathsf W^{(n)}\to \R$ by 
\begin{multline}
	\Lpart^{(n, \partial)} f(\vec x)  = \sum_{\underset{x_i-x_{i-1}>1}{2\leqslant i\leqslant n} } \p\left( f(\vec x_i^-) -f(\vec x)\right) + \sum_{   \underset{x_{i+1}-x_{i}>1}{1\leqslant i\leqslant n}} \q\left( f(\vec x_i^+) -f(\vec x)\right)\\ 
	+\mathds{1}_{x_1>1}\p\left( f(\vec x_1^-) -f(\vec x)\right) + \mathds{1}_{x_1=1} (\p-\q)\rho\left( f(\partial) - f(\vec x)\right),
	\label{eq:defLpartdagger}
\end{multline}
when $\mathsf x=\vec x\in \Weyl{n}_{\geq 1}$ and $\Lpart^{(n, \partial)} f(\mathsf x)=0$ if $\mathsf x=\partial$. Define the function 
$$ \mathsf H(\eta; \mathsf x):= \begin{cases}
	H(\eta; \vec x) &\mbox{ if }\mathsf x=\vec x\in \Weyl{n}_{\geq 1}, \\
	0 &\mbox{ if }\mathsf x=\partial. 
\end{cases} $$
\begin{corollary}
	Assume that \eqref{eq:Liggettcondition} holds. For any $\mathsf x\in \mathsf W^{(n)}$ and $\eta\in \lbrace 0,1\rbrace^{\Z_{>0}}$, we have 
	\begin{equation}
		\mathbb E^{\eta}\left[\mathsf H(\eta(t), \mathsf x) \right] = \mathbb E_{\mathsf x}\left[\mathsf H(\eta, \mathsf x(t)) \right],
		\label{eq:markovdualityhalfspace}
	\end{equation}
where in the left-hand-side $\mathbb E^{\eta}$ denotes the expectation with respect to the half-line ASEP (i.e. the Markov process $\eta(t)$ on $\lbrace 0,1\rbrace^{\Z_{>0}}$ with generator \eqref{eq:defLeta})  starting from initial condition $\eta$, and in the right-hand-side, $\mathbb E_{\mathsf x}$ denotes the expectation with respect to the process $\mathsf x(t)$ on $\mathsf W^{(n)}$ with generator  \eqref{eq:defLpartdagger}, starting from initial condition $\mathsf x$.
	\label{cor:markovdualityhalfspace}
\end{corollary}
\begin{proof}
	Let us show that for any $\mathsf x\in \mathsf W^{(n)}$ and $\eta\in \lbrace 0,1\rbrace^{\Z_{>0}}$, 
	\begin{equation}
		\Leta \mathsf H(\eta; \mathsf x) = \Lpart^{(n, \partial)} \mathsf H(\eta; \mathsf x).
		\label{eq:markovdualitygenerators}
	\end{equation}
When $\mathsf x=\partial$, \eqref{eq:markovdualitygenerators} holds since both sides equal $0$. Assume now that $\mathsf x=\vec x\in \Weyl{n}_{\geq 1}$. In that case, $\Leta \mathsf H(\eta; \mathsf x)=\Leta  H(\eta; \vec x)$ by definition of the function $\mathsf H$, and since $\mathsf H(\eta; \partial)=0$, the definition of $\Lpart^{n, \partial}$ in \eqref{eq:defLpartdagger} shows that $\Lpart^{(n, \partial)} \mathsf H(\eta; \vec x) = \Lpart^{(n,\rho)} H(\eta; \vec x)$. Thus, Theorem \ref{theo:dualityhalf-space} implies \eqref{eq:markovdualitygenerators}. Finally, since $\mathsf H$ is bounded, the duality of generators implies \eqref{eq:markovdualityhalfspace} (see \cite[Prop. 1.2]{jansen2014}). 
\end{proof}

\subsection{A system of ODEs for half-line ASEP} 
\label{sec:ODE}
In this section, we show how the  duality from Theorem \ref{theo:dualityhalf-space} can be used to find a closed system of ODEs that characterize the function of $(t,\vec x)\mapsto \mathbb E[\prod_{i=1}^n Q_{x_i}(\eta(t))]$. 
\begin{proposition} Assume that \eqref{eq:Liggettcondition} holds and fix some initial state $\eta\in \lbrace 0,1\rbrace^{\Z_{>0}}$ with finitely many particles. 
There exists a unique function $ u : \R_{+} \times \Weyl{n}_{\geq 1}  \to \R$ which satisfies 
\begin{enumerate}
	\item For all $\vec x \in \Weyl{n}_{\geq 1}$ and $t\in \R_+$, 
	\begin{equation} \frac{d}{dt} u(t;\vec x) = \Lpart^{(n, \rho)}  u(t;\vec x), 
	\label{eq:cond1}
	\end{equation}
where we recall that $\Lpart^{(n, \rho)}$ is defined in \eqref{eq:defLpartboundary}; 
	\item For any $T>0$, there exist constants $c,C>0$ such that for all $t\in [0,T]$ and $\vec x \in \Weyl{n}_{\geq 1}$, 
	\begin{equation}
\vert u(t; \vec x) \vert \leqslant C \prod_{i=1}^n  e^{c\vert x_i\vert}.
\label{eq:boundedgrowth}
	\end{equation} 
	\item For any $\vec x\in \Weyl{n}_{\geqslant 1}$, $u(0; \vec x)=H(\eta,\vec x)$.
\end{enumerate}
The solution to (1),(2) and (3) above is such that for all  $\vec x \in \Weyl{n}_{\geqslant 1}$ and $t\in \R_+$, $u(t,\vec x) = \EE^{\eta}[H(\eta(t),\vec x))]$, where $\mathbb E^{\eta}$ denotes the expectation with respect to half-line ASEP with initial condition $\eta$.
\label{prop:trueevolution}
\end{proposition}
\begin{proof}
	We will show that the function $(t,\vec x)\mapsto \EE^{\eta}[H(\eta(t),\vec x))]$ is the unique solution to (1), (2) and (3). We have that 
	\begin{align*}
		\frac{d}{dt} \mathbb E^{\eta} \left[H(\eta(t), \vec x) \right] &= \Leta \mathbb E^{\eta} \left[H(\eta(t), \vec x) \right] \\
		&=  \mathbb E^{\eta} \left[\Leta H(\eta(t), \vec x) \right]\\
		&= \mathbb E^{\eta} \left[ \Lpart^{(n, \rho)}  H(\eta(t), \vec x)\right]\\
		&= \Lpart^{(n, \rho)} \mathbb E^{\eta} \left[ H(\eta(t), \vec x) \right].
	\end{align*}
	The first equality holds by definition of the generator, the second uses the commutation between generator and semi-group, the third uses Theorem \ref{theo:dualityhalf-space} and the last equality holds by linearity of the expectation and the fact that  $\Lpart^{(n, \rho)} H(\eta(t), \vec x)$ is simply a finite sum. 
	This implies that $\mathbb E^{\eta} \left[H(\eta(t), \vec x) \right]$ must satisfy conditions (1) of the statement of the Proposition. The condition (2) is obviously satisfied because $H(\eta, \vec x)<1$. The condition (3) is satisfied by definition (the initial condition of the process $\eta(t)$ is $\eta$). 
	
Let us now turn to the uniqueness. Assume that we have two solutions $u_1$ and $u_2$ with the same initial data and let $g=u_1-u_2$.  Let $\mathbb E_{\vec x}$ denote the expectation with respect to the sub-Markov process $\vec x(t)$ with initial condition $\vec x$ and generator $\Lpart^{(n, \rho)}$. Using  \eqref{eq:cond1} and Kolmogorov's backward equation,  for all $\vec x \in \Weyl{n}_{\geq 1}$, 
\begin{equation}
\frac{d}{dt} \mathbb E_{\vec x}\left[ g(t,\vec x(T-t)) \right] = 0.
\label{eq:derivative0}
\end{equation}
Thus, the function is constant, and comparing the values at $t=0$ and $t=T$, we obtain that $g(t,\vec x)=0$ for all $\vec x \in \Weyl{n}_{\geq 1}$, hence the uniqueness. We refer to \cite[Appendix C]{borodin2012duality} and \cite[Proposition 3.6]{barraquand2016q} for details about why $\mathbb E_{\vec x}\left[ g(t,\vec x (T-t)) \right]$ is well-defined, continuous and differentiable and how the growth hypothesis \eqref{eq:boundedgrowth} is used in proving so. 
\end{proof}
Using Proposition \ref{prop:freeevolutionu}, we can reformulate the duality of operators in Theorem \ref{theo:dualityhalf-space} as an equality of expectations (different from Corollary \ref{cor:markovdualityhalfspace}). 
\begin{corollary}
Assume that \eqref{eq:Liggettcondition} holds and fix $\vec x\in \Weyl{n}_{\geq 1}$ and some initial state $\eta\in \lbrace 0,1\rbrace^{\Z_{>0}}$ with finitely many particles. Then, 
\begin{equation}
\EE^{\eta}\left[H(\eta(t), \vec x) \right]  = \mathbb E_{\vec x}\left[H(\eta, \vec x(t)) e^{-\rho(\p-\q) \int_0^t \mathds{1}_{x_1(s)=1} ds} \right],	
\label{eq:dualityreweighted}
\end{equation}
 where $\mathbb E^{\eta}$ denotes the expectation with respect to half-line ASEP with initial condition $\eta$, and $\mathbb E_{\vec x}$ denotes the expectation with respect to the Markov process $\vec x(t)$ on $\Weyl{n}_{\geq 1}$ with generator $\Lpart^{(n, 0)}$, that is the ASEP on $\mathbb Z_{>0}$ with closed boundary conditions and jump rates $\p$ to the left and $\q$ to the right. 
 \label{cor:halfspace}
\end{corollary}
\begin{proof}
By Proposition \ref{prop:freeevolutionu}, it suffices to check that the function 
\begin{equation*}
	f(t;\vec x):= \EE_{\vec x}\left[ H(\eta, \vec x(t)) e^{-\rho(\p-\q) \int_0^t \mathds{1}_{x_1(s)=1} ds} \right]
\end{equation*} 
 satisfies the three conditions in Proposition \ref{prop:freeevolutionu}. Observe that $\Lpart^{(n, \rho)} = \Lpart^{(n, 0)} + V^{(\rho)} $ where the function $V^{(\rho)}:\Weyl{n}_{\geq 1}\to \mathbb R$ is $V^{(\rho)}(\vec x) = -\rho (\p-\q) \mathds{1}_{x_1=0}$. Since $V^{(\rho)}$ and $H$ are bounded, the condition (2) is clearly satisfied. Condition (3) is satisfied as well, since we assumed that the process $\vec x(t)$ starts from $\vec x(0)=\vec x$. To check condition (1), we first write 
\begin{equation*}
f(t;\vec x) = \EE_{\vec x(-t)=\vec x}\left[H(\eta, \vec x(0)) e^{ \int_{-t}^0 V^{(\rho)}(\vec x(s)) ds}  \right].
\end{equation*}
Then, 
\begin{multline*}
	\frac{d}{dt}  f(t;\vec x) = \\  \Lpart^{(n, 0)} \EE_{\vec x(-t)=\vec x}\left[   H (\eta, \vec x(0)) e^{ \int_{-t}^0 V^{(\rho)}(\vec x(s)) ds}  \right]  + \EE_{\vec x(-t)=\vec x}\left[V^{(\rho)}(\vec x(-t))  H(\eta, \vec x(0)) e^{ \int_{-t}^0 V^{(\rho)}(\vec x(s))ds} \right]=\\
	  \Lpart^{(n, \rho)} f(t;\vec x).
\end{multline*}
In the first equality we have used Kolmogorov's backward equation. In the second equality, we use the initial condition $\vec x(-t)=\vec x$ to recognize the action of $\Lpart^{(n, \rho)}$. Hence, $f(t;\vec x)$ satisfies the three conditions in Proposition  \ref{prop:freeevolutionu}, which concludes the proof. Alternatively, Condition (3) can be viewed as an instance of the Feynman-Kac formula for a continuous time Markov process on a countable state space with bounded potential $V^{(\rho)}$. In particular, Corollary \ref{cor:halfspace} can be deduced from Theorem \ref{theo:dualityhalf-space} using \cite[Corollary 4.13]{ethier2009markov} (see also \cite[Eq. (48)]{kesten1986influence}). 
\end{proof} 
\begin{remark}
The factor $e^{-\rho(\p-\q) \int_0^t \mathds{1}_{x_1(s)=1} ds}$ in \eqref{eq:dualityreweighted} can be interpreted as the probability that the killed $n$-particle half-line ASEP discussed above is not killed up to time $t$ (or, equivalently, the probability that $\mathsf x(t)\neq\partial$, in the setting of Section \ref{sec:markovdualityhalfspace}). Indeed, the process $\mathsf x(t)$ can be constructed as follows. Let $\vec x(t)$ be the Markov process  on $\Weyl{n}_{\geq 1}$ with generator $\Lpart^{(n, 0)}$ as in Corollary \ref{cor:halfspace} and let $(\mathcal F_t)_{t\geq 0}$ denote the associated filtration. Let $\mathcal E$ be an independent exponential random variable with rate $\rho(\p-\q)$, and define a local time $L(t)=\int_0^t\mathds{1}_{x_1(s)=1}ds$. Then, for any $t$ such that $L(t)<\mathcal E$, $\mathsf x(t)=\vec x(t)$, and when the clock rings, that is when $L(t)=\mathcal E$, the process jumps to $\mathsf x(t)=\partial$ where it stays forever. Conditionally on $\mathcal F_t$, $\mathbb P(\mathcal E > L(t)) = e^{-\rho(\p-\q) L(t)}$. Since $\mathsf H$ vanishes when $\mathsf x=\partial$
$$ \mathbb E_{\mathsf x}\left[\mathsf H(\eta, \mathsf x(t)) \right] =  \mathbb E_{\vec x}\left[  \mathbb E_{\mathsf x}\left[ \mathsf H(\eta; \mathsf x(t)) \vert \mathcal F_t \right]  \right]  = \mathbb E_{\vec x}\left[ H(\eta; \vec x(t))e^{-\rho(\p-\q) L(t)} \right],$$ 
so that Corollary \ref{cor:halfspace} and Corollary \ref{cor:markovdualityhalfspace} are equivalent. 
\label{rem:killing}
\end{remark}


It is  not a priori clear how to find an explicit formula solving the system of equations from Proposition \ref{prop:trueevolution}. We now use a reformulation of ASEP's Kolmogorov equation as a system of ODEs on a larger space subject to boundary conditions. This was already used in solving full-line ASEP via  Bethe ansatz in  \cite{schutz1997exact, tracy2008integral, borodin2012duality} and half-line ASEP with closed boundary conditions \cite{tracy2013bose}.
\begin{proposition} Assume that \eqref{eq:Liggettcondition} holds and fix some initial state $\eta\in \lbrace 0,1\rbrace^{\Z_{>0}}$ with finitely many particles. . 
	If $ u : \R_{+} \times \Z_{\geqslant 0}^n \to \R$ solves 
	\begin{enumerate}
		\item For all $\vec x \in \Z_{\geqslant 1}^n$ and $t\in \R_+$, 
		\begin{equation*} \frac{d}{dt} u(t;\vec x) = \Delta^{\p, \q}  u(t;\vec x), \end{equation*}
		where 
		\begin{equation}\Delta^{\p,\q} f(\vec x) = \sum_{i=1}^n \p f(\vec x_i^-) +\q f(\vec x_i^+) -(\p + \q)f(\vec x);
		\label{eq:defDelta}
	\end{equation}
		\item For all $\vec x \in \Z_{\geqslant 1}^n$ such that for some $i\in \lbrace 1, \dots,  n-1\rbrace $,  $ x_{i+1}=x_i+1$, we have 
		\begin{equation*} \p u(t; \vec x_{i+1}^-) + \q  u(t; \vec x_i^+) = (\p+\q)  u(t; \vec x);\end{equation*}
		\item For all $t\in \R_+$ and $(x_2,  \dots, x_n)\in \Weyl{n-1}_{\geqslant 2}$, 
		\begin{equation}
		u(t;0,x_2,\dots) = (\rho q+1-\rho)u(t;1,x_2,\dots);
		\label{eq:boundaryconditionforu}
		\end{equation}
		\item For any $T>0$, there exist constants $c,C>0$ such that for all $t\in [0,T]$ and $\vec x \in \Weyl{n}_{\geqslant 1}$, 
		\begin{equation*} \vert u(t; \vec x) \vert \leqslant C\prod_{i=1}^ne^{c\vert x_i\vert};\end{equation*}
		\item For any $\vec x\in \Weyl{n}_{\geqslant 1}$, $u(0; \vec x)=H(\eta,\vec x)$.
	\end{enumerate}
	Then for all  $\vec x \in \Weyl{n}_{\geqslant 1}$ and $t\in \R_+$, $u(t,\vec x) = \EE^{\eta}[H(\eta(t),\vec x))]$, where $\mathbb E^{\eta}$ denotes the expectation with respect to half-line ASEP with initial condition $\eta$.
	\label{prop:freeevolutionu}
\end{proposition}
\begin{proof}
Assume that $u$ satisfies conditions (1) to (5). The conditions (1) and (2) are a  well-known reformulation of ASEP's generator, see \cite[Section 2]{schutz1997exact} (see also \cite[Eq. (2.2)]{tracy2008integral} or  \cite[Proposition 4.10]{borodin2012duality} for instance). It ensures that for all $\vec x\in \Weyl{n}_{\geq 1}$, 
	\begin{equation*}
\frac{d}{dt} u(t;\vec x) = \Lpart^{(n)}  u(t;\vec x).
	\end{equation*}
Then, the boundary condition \eqref{eq:boundaryconditionforu} when $x_1=1$ ensures that for all $\vec x\in \Weyl{n}_{\geq 1}$, $\Lpart^{(n)}  u(t;\vec x)=\Lpart^{(n, \rho)}  u(t;\vec x)$. Thus condition (1) of Proposition \ref{prop:trueevolution} is satisfied. 
The conditions (4) and (5) above are the same as the conditions (2) and (3) from Proposition \ref{prop:trueevolution}. Hence we may apply Proposition \ref{prop:trueevolution} to conclude the proof. 
\end{proof}

\subsection{Exact integral formulas} In this section, we provide an integral formula for the function $v_n(t;\vec x) = \EE^{\emptyset}\left[ Q_{x_1}(\eta(t))\dots Q_{x_n}(\eta(t)) \right]$ for half-line ASEP with empty initial condition. To do that, our goal is to find a function that satisfies the conditions of Proposition \ref{prop:freeevolutionu} with  $v_n(0;\vec x)=1$, corresponding to the empty initial condition.  
\label{sec:formula}
\subsubsection{Comparison with earlier work} For the full space ASEP, an analogue of Proposition \ref{prop:freeevolutionu} was proved in \cite{borodin2012duality} and involves the same conditions except condition (3). Inspired by the coordinate Bethe ansatz \cite{imamura2011current} and by the theory of Macdonald processes \cite{borodin2014macdonald}, an ansatz was proposed in \cite{borodin2012duality} to find solutions of the conditions (1) (2) (4) and (5) of Proposition \ref{prop:freeevolutionu}. The idea is to look for functions $v_n(t; \vec x)$  that take the form of a function of $n$ complex variables with some specific structure,  integrated over contours in the complex plane that are nested into each other. This nested contour integral ansatz was later shown to apply to a variety of other models \cite{corwin2013q, corwin2014q, borodin2016stochasticsix, barraquand2016q, barraquand2017random}. For half-space models, a variant of the nested contour integral ansatz was suggested in \cite{borodin2016directed} to solve the half-line delta Bose gas, and formulas with a similar structure were proved for a model of random walks in random environment in a half-space \cite{barraquand2022random}. The formulas in these two works take the same form as moments of half-space Macdonald processes \cite{barraquand2018half}. In particular, the  $q$-moments of the height function of the half-space stochastic six vertex model can be computed as reasonably simple nested contour integral formulas \cite[Section 5]{barraquand2018half}. Since the half-space stochastic six vertex model considered in \cite{barraquand2018half} converges to the half-line open ASEP with $\rho=1/2$ (as shown in \cite{barraquand2018stochastic}), it is reasonable to expect that $q$-moments of the half-line open ASEP can also be written as nested contour integrals, at least when $\rho=1/2$. Unfortunately, this is not the case, the nested contour integral ansatz does not work for the half-line open ASEP, even for $\rho=1/2$. This issue was already anticipated in \cite{barraquand2018half}, which suggested that after expanding  the moment formulas for the half-space six vertex model as a sum over certain residues, the formulas might have a meaningful ASEP limit. 

Thus, partly inspired by the formulas from \cite[Section 5]{barraquand2018half}, and partly inspired by residue expansions arising in \cite[Section 7.3]{borodin2016directed}, we will define the function $v_n(t;\vec x)$ as a sum over residues instead of a simple nested contour integral, and check that it satisfies the conditions of Proposition \ref{prop:freeevolutionu}. The structure of the residue expansion is quite intricate, but we will see that in the KPZ equation limit in Section \ref{sec:KPZ}, the sum over residues can be rewritten as a relatively simple contour integral, as predicted in \cite{borodin2016directed}. 

\subsubsection{Residue subspaces}
Before stating our formula for $v_n(t;\vec x)$ we need to introduce some notations to define the residues that arise in the formula. The residues will be indexed by integer partitions and diagrams that we define below. We recall that an integer partition is a  non-increasing finite sequence of positive integers. We say that the partition $\lambda=(\lambda_1, \lambda_2, \dots, \lambda_{\ell(\lambda)})$ is a partition of $n$, denoted $\lambda \vdash n$,  if $\sum_{i=1}^{\ell(\lambda)}\lambda_i =n$. The number $\ell(\lambda)$ is called the length of the partition $\lambda$. For a given partition $\lambda\vdash n$,  we consider diagrams formed by  $\ell(\lambda)$ lines of numbers separated by arrows, of the following form:  
\begin{eqnarray*}
	i_1 &\rightarrowplus i_2\rightarrowplus \dots \rightarrowplus i_{\mu_1-1} \rightarrowplus &i_{\mu_1}\leftarrowminus i_{\mu_1+1} \leftarrowplus \dots\leftarrowplus i_{\lambda_1}\\
		j_1 &\rightarrowplus j_2\rightarrowplus \dots \rightarrowplus j_{\mu_2-1} \rightarrowplus &j_{\mu_2}\leftarrowminus j_{\mu_2+1} \leftarrowplus \dots\leftarrowplus j_{\lambda_2} \\
	&&	\vdots   
\end{eqnarray*}
The diagrams are such that 
\begin{itemize}
	\item all integers $i_1,i_2,\dots, j_1,j_2,\dots$ are distinct elements of  $\lbrace 1, \dots,n\rbrace$, so that every integer in $\lbrace 1, \dots,n\rbrace$ appears exactly once in the diagram; 
	\item all directed arrows go from a larger integer to a smaller one; 
	\item on each line, arrows change from pointing right to pointing left at most once, i.e., the sequence of numbers is decreasing, until a certain minimum, and then it is increasing. We denote by $\mu_i$ the index of the minimal number on the $i$th line, around which arrows change orientation. Note that when $\mu_i=1$, all arrows are pointing left, and when $\mu_i=\lambda_i$, all arrows are pointing right;  
	\item all arrows bear a plus sign, except for the first arrow pointing left in each row (when there is one such arrow) which bears a minus sign.
\end{itemize}
\begin{example}
For the partition $\lambda=(4,3,1,1)$, there exists many possible diagrams, such as the following three diagrams:

\begin{minipage}[center]{5cm}
\begin{align*}
	&	7 \rightarrowplus 5 \rightarrowplus 1\leftarrowminus 3\\
	&	2\leftarrowminus 4\leftarrowplus 9\\
	&	6\\
	&	7
\end{align*}
\end{minipage}
\begin{minipage}{5cm}
	\begin{align*}
		&	7 \rightarrowplus 5 \rightarrowplus 1\leftarrowminus 3\\
		&	2\leftarrowminus 4\leftarrowplus 9\\
		&	7\\
		&	6
	\end{align*}
\end{minipage}
\begin{minipage}{5cm}
	\begin{align*}
		&	7 \rightarrowplus 1 \leftarrowminus 3\leftarrowplus 5\\
		&	9\rightarrowplus 4\rightarrowplus 2\\
		&	6\\
		&	7
	\end{align*}
\end{minipage}

\noindent Note that the left and middle diagrams, which differ only by a permutation of the last two lines, are considered as different diagrams. 
\end{example}

We denote by $S(\lambda)$ the set of all such diagrams and we will generally denote by the letter $I$ one such diagram. Now, we will define a procedure to associate to any diagram $I\in S(\lambda)$ a certain residue subspace as follows. Recall that for a meromorphic function $f$ with a simple rational pole at $a$, its residue at $a$, denoted $\Res{z\to a}\lbrace f(z) \rbrace$,  is simply given by
\begin{equation*} 
	\Res{z\to a}\lbrace f(z) \rbrace = (z-a)f(z)\Big\vert_{z=a}. 
	\end{equation*} 
For a function $f(\vec z)$ of $n$ complex variables that is meromorphic in each variable, we can similarly consider its residue in the variable $z_1$ at $r(\vec z)$, where $r(\vec z)$ does not depend on $z_1$ and is a rational function of the remaining variables $z_2, \dots, z_n$ that we treat as constants. Again, assuming the pole is a simple rational pole, the residue is given by 
\begin{equation*} \Res{z_1\to r(\vec z)}\lbrace f(\vec z) \rbrace = f(\vec z) (z_1-r(z_2,\dots,z_n))\Big\vert_{z_1=r(\vec z)},
	\end{equation*}
and can be seen as a function of variables $z_2, \dots, z_n$. More generally, for indices $i_1, \dots, i_k\in \lbrace 1, \dots, n\rbrace$,  we may compute successively the residue in $z_{i_1}$ at $r_1(\vec z)$ (where $r_1(\vec z)$ is a rational function of the variables $z_j$ for $j\neq i_1$), the residue in $z_{i_2}$ at $r_2(\vec z)$ (where $r_1(\vec z)$ is a rational function of the variables $z_j$ for $j\neq i_1, i_2$), etc. until the residue in the variable $z_{i_k}$ at $r_k(\vec z)$ (where $r_k(\vec z)$ is a rational function of the variables $z_j$ for $j\neq i_1, \dots, i_k$). Assuming the poles are simple rational poles, which will be the case below, this multi-residue is given by 
\begin{equation}    
	\Res{\substack{z_{i_1}\to r_1(\vec z) \\ \dots \\ z_{i_k}\to r_k(\vec z)}}\lbrace f(\vec z) \rbrace =   f(\vec z)\prod_{j=1}^k(z_{i_j}-r_i(\vec z))\Bigg\vert_{\substack{z_{i_1}= r_1(\vec z) \\ \dots \\ z_{i_k}= r_k(\vec z)}}.
	\label{eq:residuesubspace}
	\end{equation} 
Now, for a diagram $I\in S(\lambda)$ with $\lambda\vdash n$, we will define a multi-residue of the form \eqref{eq:residuesubspace} with $k=n-\ell(\lambda)$, that is, we will compute residues in as many variables as the number of arrows in the diagram $I$, according to the following rules:   
\begin{itemize}
	\item for each any arrow of the form $a\rightarrowplus b$ or $b\leftarrowplus a$ we compute the residue in the variable $z_a$ at $qz_b$; 
	\item for each arrow of the form $a\leftarrowminus b$ we compute the residue in the variable $z_b$ at $1/z_a$.
\end{itemize} 
We will denote this  multi-residue  by $\Res{I} \left\lbrace  f(\vec z)\right\rbrace$. We may compute the residues associated to different lines in $I$ in any order. However, within a line, the residues must be computed in an appropriate order so that  \eqref{eq:residuesubspace} can be applied. Such an order can always be found, it suffices to start from the extremities of the line and follow the arrows. After computing all the residues, $\Res{I} \left\lbrace  f(\vec z)\right\rbrace$ only depends on  the variables $z_j$ for each $j$ being the minimal number on one of the lines of the diagram $I$. In other terms, $\Res{I} \left\lbrace  f(\vec z)\right\rbrace$ is a function of the variables $z_{i_{\mu_1}}, z_{j_{\mu_2}}, \dots$ There are as many variables as the length of the partition $\lambda$. 
\begin{example}For the partition $(1,1,\dots)$, there is only one possible diagram where each line has a unique number and no arrow. In that case, $\Res{I} \left\lbrace f(\vec z) \right\rbrace=f(\vec z)$. 
\end{example} 
\begin{example}For $n=7$,  $\lambda=(4,2,1)$ and 
		\begin{eqnarray*}
		&&	6 \rightarrowplus 1 \leftarrowminus 2\leftarrowplus 3\\
I=	&&	5\rightarrowplus 4\\
		&&	7
	\end{eqnarray*}
To determine  $\Res{I} \left\lbrace  f(\vec z)\right\rbrace$ one can compute successively  the residue in $z_6$ at $qz_1$ , the residue in $z_3$ at $qz_2$, the residue in $z_2$ at $1/z_1$ and finally the residue in $z_5$ at $qz_4$, and one obtains 
\begin{equation*} \Res{I} \left\lbrace  f(\vec z)\right\rbrace = (z_6-qz_1)(z_2-1/z_1)(z_3-qz_2)(z_5-qz_4)f(z_1, z_2, z_3, z_4, z_5, z_6, z_7)\Bigg\vert_{\substack{z_{6}= q z_1 \\ z_2= 1/z_1 \\ z_3 = q/z_1 \\ z_5= qz_4}}.\end{equation*}
\end{example}

\begin{remark}
	The diagrams defined in \cite{borodin2016directed} are slightly more general, allowing the arrows surrounding  $i_{\mu_1}, j_{\mu_2},\dots$ to be of the form $\rightarrowplus  i_{\mu_1} \leftarrowminus$ or $\rightarrowminus i_{\mu_1} \leftarrowplus$, but the residual subspaces associated to the latter arrows can be obtained from the former arrows after reflecting the numbers around $i_{\mu_1}$. This double counting is the reason why a factor $2^{m_2+m_3+\dots}$ is present in \cite[Eq. (45)]{borodin2016directed} but not in our equation \eqref{eq:deffunctionv} below. Note also that arrows in \cite{borodin2016directed} go from smaller to larger integers, while our arrows go from larger to smaller integers, but this is inconsequential up to reordering variables.
	\label{rem:doublecounting}
\end{remark}

\subsubsection{Formula for $\EE\left[ Q_{x_1}(\eta(t))\dots Q_{x_n}(\eta(t))  \right]$}
Let us first define a function of $n$ complex variables $\vec z=(z_1, \dots, z_n)\in \C^n$ by  
\begin{equation} \phi_{\vec x}(\vec z) = q^{\frac{n(n-1)}{2}} \prod_{i<j} \frac{z_i-z_j}{qz_i-z_j} \frac{1-q z_iz_j}{1-z_iz_j}\prod_{j=1}^n\frac{F_{x_j}(z_j)}{z_j}, 
	\label{eq:defphi}
\end{equation}
where
\begin{equation}
	F_x(z) = \frac{1-qz^2}{1-z} \exp\left( \frac{(1-q)^2 z\p t}{(1-z)(1-qz)}\right) \left( \frac{1-z}{1-q z} \right)^{x}\frac{\rho}{\rho + (1-\rho)z}.
	\label{eq:defF}
\end{equation} 
We denote by $\mathcal C$ the positively oriented circle of radius $1/\sqrt{q}$ around $0$. Using these definitions, we are now able to define a function $v_n(t;\vec x)$ as 
\begin{equation}
	v_n(t;\vec x) := \sum_{\lambda \vdash n} \frac{(-1)^{n-\ell(\lambda)}}{m_1!m_2!\dots}\sum_{I\in S(\lambda)} \oint_{\mathcal C} \frac{dz_{i_{\mu_1}}}{2\I\pi}  \oint_{\mathcal C} \frac{dz_{j_{\mu_2}}}{2\I\pi} \dots \; \Res{I} \left\lbrace \phi_{\vec x}(\vec z)\right\rbrace,
	\label{eq:deffunctionv}
\end{equation}
where $m_1, m_2, \dots $ denotes the respective  multiplicities of $1,2,\dots$ in the partition $\lambda$. 


\begin{theorem}Assume \eqref{eq:Liggettcondition}, let $\rho=\ratealpha/\p$ and assume that $\rho\in \left(\frac{1}{1+\sqrt{q}},1 \right]$.   For any $n\in \Z_{\geq 1}$,   $t\geqslant 0$ and $x\in \Weyl{n}_{\geqslant 1}$,
	\begin{equation}
		\EE\left[ Q_{x_1}(\eta(t))\dots Q_{x_n}(\eta(t))  \right]  = v_n(t;\vec x),
		\label{eq:mixedmomentsQ}
	\end{equation} 
	where $v_n$ is defined by \eqref{eq:deffunctionv}. In particular, for any $x\geq 1$, 
	\begin{equation}
		\EE[q^{N_x(t)}] = \oint_{\mathcal C} \frac{dz}{2\I\pi} \frac{F_x(z)}{z},
		\label{eq:firstmoment}
	\end{equation}
and for any $1\leq x_1<x_2$, 
\begin{multline}
	\EE[q^{N_{x_1}(t)+N_{x_2}(t)}] = \oint_{\mathcal C} \frac{dz_1}{2\I\pi}\oint_{\mathcal C} \frac{dz_2}{2\I\pi} \phi_{\vec x}(z_1,z_2) + \\ (1-q)\oint_{\mathcal C} \frac{dz}{2\I\pi} \frac{1-q^2 z^2}{1-q z^2}\frac{F_{x_1}(z)F_{x_2}(qz)}{qz} +(1-q) \oint_{\mathcal C} \frac{dz}{2\I\pi} \frac{1-z^2}{1-q z^2}\frac{F_{x_1}(z)F_{x_2}(1/z)}{z}. 
	\label{eq:explicitsecondmoment}
\end{multline}
	\label{theo:momentformula}
\end{theorem} 
The condition $\rho\in \left(\frac{1}{1+\sqrt{q}},1 \right]$ is a technical restriction. It is equivalent to  $\frac{\rho}{1-\rho}>\frac{1}{\sqrt{q}}$, which ensures that the only singularity of the function $F_x(\vec z)$, defined in \eqref{eq:defF}, inside the contour $\mathcal C$ is at $z=1$. We expect that for values of $\rho$ that do not satisfy this restriction, $v_n(t; \vec x)$ should be given by the analytic continuation in $\rho$ of \eqref{eq:deffunctionv}, but it would be quite cumbersome to write it explicitly, so we do not pursue this. 

\begin{remark}  The specific form of the function $\phi_{\vec x}(\vec z)$ is inspired from several recent works on integrable half-space probabilistic systems, namely half-line ASEP with closed boundary \cite{tracy2013bose}  and the stochastic six-vertex model in a half-quadrant \cite[section 5]{barraquand2018half} (see also \cite{borodin2016directed} and \cite{barraquand2022random} for integral formulas with a similar structure). 
\end{remark}

\subsection{Proof of Theorem \ref{theo:momentformula}} 
To prove Theorem \ref{theo:momentformula}, we will check below that the function $v_n(t; \vec x)$ defined in \eqref{eq:deffunctionv} satisfies the conditions from  Proposition \ref{prop:freeevolutionu}. 
	
\subsubsection{Condition (1)} We have 
	\begin{equation*} \Delta^{\p,\q}\left( \frac{1-z}{1-q z} \right)^{x} =   \frac{(1-q)^2 z \p}{(1-z)(1-qz)}  \left( \frac{1-z}{1-q z} \right)^{x}, 
	\end{equation*}
where $\Delta^{\p,\q}$ was defined in \eqref{eq:defDelta} and acts on functions of the variable $x$. 
	Hence, we have that  $\frac{d}{dt}  \phi_{\vec x}(\vec z) = \Delta^{\p,\q}\phi_{\vec x}(\vec z)$, and by linearity of the integral and residues, we conclude that   
	\begin{equation*}\frac{d}{dt} v_n(t; \vec x) = \Delta^{\p,\q}v_n(t; \vec x).\end{equation*}
	
\subsubsection{Condition (2)} When $x_i+1=x_{i+1}$, 
	\begin{equation*}  \p v_n(t; \vec x_{i+1}^-) + \q  v_n(t; \vec x_i^+) - (\p+\q)  v_n(t; \vec x) \end{equation*}
	has the same expression as $v_n(t; \vec x_{i+1}^-)$ in \eqref{eq:deffunctionv}, with an extra factor 
		\begin{equation*} \p + \q \left( \frac{1-z_{i}}{1-qz_{i}}\right)\left(\frac{1-z_{i+1}}{1-qz_{i+1}}\right)-(\p+\q)\left(\frac{1-z_{i+1}}{1-qz_{i+1}}\right) =\frac{(\p-\q)(qz_i-z_{i+1})}{(1-qz_{i})(1-qz_{i+1})}
		\end{equation*}
	in the function $\phi_{\vec x}$.  This extra factor brings poles at $z_i=1/q$ and $z_{i+1}=1/q$ located outside the contour $\mathcal C$, and it brings  a factor $(qz_i-z_{i+1})$ which  cancel the one present in the denominator.
	More precisely, we write 
	\begin{equation}
		\p v_n(t; \vec x_{i+1}^-) + \q  v_n(t; \vec x_i^+) - (\p+\q)  v_n(t; \vec x) = \sum_{\lambda \vdash n} \frac{(-1)^{n-\ell(\lambda)}}{m_1!m_2!\dots}\sum_{I\in S(\lambda)} \oint_{\mathcal C} \frac{dz_{i_{\mu_1}}}{2\I\pi}  \oint_{\mathcal C} \frac{dz_{j_{\mu_2}}}{2\I\pi} \dots \; \Res{I} \left\lbrace \widetilde \phi(\vec z)\right\rbrace, 
		\label{eq:sumphitilde}
	\end{equation} 
	where 
	\begin{equation} \widetilde \phi(\vec z) =\frac{(\p-\q)(qz_i-z_{i+1})}{(1-qz_{i})(1-qz_{i+1})} \phi_{\vec x_{i+1}^-} (\vec z).
\label{eq:defphitilde}
	\end{equation}
Our goal is to show that \eqref{eq:sumphitilde} equals $0$. We will show that for each $\lambda$, the term indexed by $\lambda$ in \eqref{eq:sumphitilde} vanishes. For each $\lambda$, we will show that for 
some diagrams $I\in S(\lambda)$, the term indexed by $I$ in \eqref{eq:sumphitilde} equals $0$, while for other diagrams, we will show that their contribution exactly cancels that of another diagram. More precisely, we will use several times the following argument. For a certain subset of diagrams  $\mathcal I \subset S(\lambda)$, we will introduce an involution $\sigma :\mathcal I\to\mathcal I$ on diagrams such that 
\begin{equation}
 \oint_{\mathcal C} \frac{dz_{i_{\mu_1}}}{2\I\pi}  \oint_{\mathcal C} \frac{dz_{j_{\mu_2}}}{2\I\pi} \dots \;	\Res{I} \left\lbrace \widetilde \phi(\vec z)\right\rbrace = - \oint_{\mathcal C} \frac{dz_{i_{\mu_1}}}{2\I\pi}  \oint_{\mathcal C} \frac{dz_{j_{\mu_2}}}{2\I\pi} \dots \;\Res{\sigma(I)} \left\lbrace \widetilde \phi(\vec z)\right\rbrace.
 \label{eq:involutioneffect}
\end{equation}
Since an involution is in particular a bijection, we may write that 
\begin{equation*}
	\sum_{I\in \mathcal I} \oint_{\mathcal C} \frac{dz_{i_{\mu_1}}}{2\I\pi}  \oint_{\mathcal C} \frac{dz_{j_{\mu_2}}}{2\I\pi} \dots \;	\Res{I} \left\lbrace \widetilde \phi(\vec z)\right\rbrace = 	\sum_{I\in \mathcal I} \oint_{\mathcal C} \frac{dz_{i_{\mu_1}}}{2\I\pi}  \oint_{\mathcal C} \frac{dz_{j_{\mu_2}}}{2\I\pi} \dots \;	\Res{\sigma(I)} \left\lbrace \widetilde \phi(\vec z)\right\rbrace,
\end{equation*}
so that using \eqref{eq:involutioneffect}, 
\begin{equation}
	\sum_{I\in \mathcal I} \oint_{\mathcal C} \frac{dz_{i_{\mu_1}}}{2\I\pi}  \oint_{\mathcal C} \frac{dz_{j_{\mu_2}}}{2\I\pi} \dots \;	\Res{I} \left\lbrace \widetilde \phi(\vec z)\right\rbrace = 	-\sum_{I\in \mathcal I} \oint_{\mathcal C} \frac{dz_{i_{\mu_1}}}{2\I\pi}  \oint_{\mathcal C} \frac{dz_{j_{\mu_2}}}{2\I\pi} \dots \;	\Res{I} \left\lbrace \widetilde \phi(\vec z)\right\rbrace=0. 
	\label{eq:conclusioninvolution}
\end{equation}
Another way to interpret \eqref{eq:involutioneffect} is that in the sum \eqref{eq:conclusioninvolution}, the terms corresponding to a diagram $I$ and its image $\sigma(I)$ exactly cancel each other, and when $I=\sigma(I)$, \eqref{eq:involutioneffect} shows that the corresponding term is the opposite of itself, hence equal to zero. 

Fix some partition $\lambda\vdash n$. It is convenient to partition the set of diagrams $S(\lambda)$ into the following four cases:
	 \begin{enumerate}
	 	\item The number $i$ and $i+1$ occur in the diagram $I$ in neighbouring positions on the same line, separated by an arrow bearing a plus sign, that is,  the diagram $I$ contains $ i+1 \rightarrowplus i $ or $i\leftarrowplus i+1$. 
	 	\item The number $i$ and $i+1$ occur in the diagram $I$ in neighbouring positions on the same line, separated by an arrow bearing a minus sign, that is,  the diagram $I$ contains $i\leftarrowminus i+1$. 
	 	\item The numbers $i$ and $i+1$ occur in the diagram $I$ in the same line but not in neighbouring positions, in which case they are necessarily separated by an arrow bearing a minus sign. 
	 	\item The numbers $i$ and $i+1$ belong to two different lines in the diagram $I$. 
	 \end{enumerate}
 We treat these four cases separately below. 
 
\textbf{Case (1):} By \eqref{eq:defphitilde}, the function $\widetilde\phi$ has no singularity when $z_{i+1}=qz_i$, so that for any diagram $I$ containing $ i+1 \rightarrowplus i $ or $i\leftarrowplus i+1$, then $\Res{I} \left\lbrace \widetilde \phi(\vec z)\right\rbrace=0$.

\textbf{Case (2):} Any diagram $I$ in that case contains $i\leftarrowminus i+1$. For such a diagram $I$, the line in the diagram $I$ containing the arrow $i\leftarrowminus i+1$ takes the form 
\begin{equation}
	A\rightarrowplus i\leftarrowminus i+1 \leftarrowplus B,
	\label{eq:snippet3}
\end{equation}
where $A$ and $B$ are (possibly empty) sequences of numbers separated by arrows that we call snippets. To this diagram $I$, we associate the diagram $\sigma(I)$ obtained from $I$ by changing the line \eqref{eq:snippet3} to 
\begin{equation*}
	B'\rightarrowplus i\leftarrowminus i+1 \leftarrowplus A',
	\label{eq:snippet4}
\end{equation*}
where $A'$ (resp $B'$) is obtained from $A$ (resp. B) by reversing the direction of arrows and reordering the numbers between them. 
\begin{example}Assume that $i=1$ and consider the diagram  
	\begin{eqnarray*}
I= 	&&	8\rightarrowplus 6 \rightarrowplus 1 \leftarrowminus 2\leftarrowplus 3 \leftarrowplus 7\\
		&&	5\rightarrowplus 4.
\end{eqnarray*}
In this case, the snippets $A$ and $B$ in \eqref{eq:snippet3} are $A = 8\rightarrowplus 6 $ and $B= 3 \leftarrowplus 7$, so that $A' = 6\leftarrowplus 8 $,  $B'= 7 \rightarrowplus 3$ and 
	\begin{eqnarray*}
	\sigma(I)= 	&&	7\rightarrowplus 3 \rightarrowplus 1 \leftarrowminus 2\leftarrowplus 6 \leftarrowplus 8\\
	&&	5\rightarrowplus 4.
\end{eqnarray*}
\end{example}

It is clear that $\sigma$ is an involution of the set of diagrams belonging to case (2).  Consider then the expression $\Res{I} \left\lbrace \widetilde \phi(\vec z) \right\rbrace$, as a function of the variable $z_i$. 
Due to the form of the function $\widetilde\phi$ in \eqref{eq:formofphitilde}, we see that by the change of variables $z_i\to 1/z_i$, 
\begin{equation}
	\oint_{\mathcal C} \frac{dz_i}{2\I\pi} \Res{I} \left\lbrace \widetilde \phi(\vec z) \right\rbrace  = - \oint_{\mathcal C'} \frac{dz_i}{2\I\pi} \Res{\sigma(I)} \left\lbrace \widetilde \phi(\vec z) \right\rbrace,
	\label{eq:equalityofresidues}
\end{equation} 
where the contour $\mathcal C'$ is a positively oriented circle of radius $\sqrt{q}$, and recall that in the right-hand-side, $\Res{\sigma(I)} \left\lbrace \widetilde \phi(\vec z) \right\rbrace$ is as a function of the variable $z_{i}$. We will now show that we can deform the contour $\mathcal C'$ to $\mathcal C$, so that 
\begin{equation}
	\oint_{\mathcal C} \frac{dz_i}{2\I\pi} \Res{I} \left\lbrace \widetilde \phi(\vec z) \right\rbrace  = - \oint_{\mathcal C} \frac{dz_i}{2\I\pi} \Res{\sigma(I)} \left\lbrace \widetilde \phi(\vec z) \right\rbrace. 
	\label{eq:equalityofresidues2}
\end{equation}  
Before proving that this deformation of contours is valid, let us illustrate how \eqref{eq:equalityofresidues} works through the following example. 
\begin{example}
\label{ex:invarianceinversion}	
Let us consider the minimal example of the case  $n=2$, $\lambda=(2)$, and $I=1\leftarrowminus 2$. In this case, $\sigma(I)=I$. We have, letting $x=x_1=x_2-1$, 
\begin{equation}
	\widetilde \phi(z_1,z_2) =q\frac{(\p-\q)(qz_1-z_{2})}{(1-qz_{1})(1-qz_{2})} \frac{z_1-z_2}{qz_1-z_2} \frac{1-q z_1z_2}{1-z_1z_2}\frac{F_{x}(z_1)}{z_1}\frac{F_{x}(z_2)}{z_2}, 
	\label{eq:specialphitilde}
\end{equation}
so that 
\begin{equation*}
	\Res{z_{2}\to 1/z_1} \lbrace 	\widetilde \phi(\vec z)\rbrace  = \frac{-q(\p-\q)(1-q)(z_1-1/z_1)}{(1-qz_{1})(1-q/z_{1})}  \frac{F_{x}(z_1)F_{x}(1/z_1)}{z_1}.
	\label{eq:specialresiduephitilde}
\end{equation*}
Due to the factor $(z_1-1/z_1)$, we see that under the change of variables $z_1\to 1/z_1$, 
\begin{equation}
	\oint_{\mathcal C} \frac{dz_1}{2\I\pi} \Res{z_{2}\to 1/z_1} \lbrace 	\widetilde \phi(\vec z)\rbrace  = - \oint_{\mathcal C'} \frac{dz_1}{2\I\pi} \Res{z_{2}\to 1/z_1} \lbrace 	\widetilde \phi(\vec z)\rbrace,
	\label{eq:specialequalityresidue}
\end{equation}
which is a special case of \eqref{eq:equalityofresidues}. 
Furthermore, we notice that 
\begin{multline}  F_{x}(z) F_{x}(1/z) = \\  \frac{1-qz^2}{1-z}\frac{1-q/z^2}{1-1/z} \exp\left( \frac{(1-q)^2 \p t(1+q)z}{(1-qz)(q-z)}\right) \left( \frac{1-z}{1-q z} \right)^{x}\left( \frac{1-1/z}{1-q /z} \right)^{x}\frac{\rho}{\rho + (1-\rho)z}\frac{\rho}{\rho + (1-\rho)/z}.
	\label{eq:simplification}
\end{multline}
While both $F_{x}(z)$ and $ F_{x}(1/z)$ have an essential singularity at $z=1$, in \eqref{eq:simplification}, the  exponential term in the product no longer has a singularity at $z=1$. Moreover, the denominators $(1-z)$ and $(1-1/z)$ in \eqref{eq:simplification} are canceled by the factors $\left( \frac{1-z}{1-q z} \right)^{x}\left( \frac{1-1/z}{1-q /z} \right)^{x}$ since $x\geq 1$. Thus,  \eqref{eq:specialresiduephitilde} has no singularity at $z=1$, so that one may deform the contour $\mathcal C'$ to $\mathcal C$ in \eqref{eq:specialequalityresidue}, and conclude that 
\begin{equation*}
	\oint_{\mathcal C} \frac{dz_1}{2\I\pi} \Res{z_{2}\to 1/z_1} \lbrace 	\widetilde \phi(\vec z)\rbrace  = - \oint_{\mathcal C} \frac{dz_1}{2\I\pi} \Res{z_{2}\to 1/z_1} \lbrace 	\widetilde \phi(\vec z)\rbrace= 0.
\end{equation*} 
\end{example}

Going back to the general case, since the function $F_x(z)$ has singularities at $z=1$, $z=1/q$ and $z=-\rho/(1-\rho)$, \eqref{eq:simplification} implies that for any nonnegative integers $r,s$, the function  
\begin{equation}
	F_x(q^rz)\dots F_x(qz)F_{x}(z) F_{x}(1/z)F_x(q/z)\dots F_x(q^s/z)
	\label{eq:nosingularitiesannulus}
\end{equation}
has no singularities at $z=1$. Letting $\omega=-\rho/(1-\rho)$, we see that \eqref{eq:nosingularitiesannulus} has singularities only at the points
\begin{equation*} q^{s+1}, q^s, \dots ,  q^2, q, q^{-1}, q^{-2} \dots, q^{-r-1} \text{ and } \omega q^{-r}, \dots, \omega, 1/\omega, \dots, q^s/\omega.\end{equation*}
In particular, since $\vert \omega\vert >1/\sqrt{q}$ by hypothesis, \eqref{eq:nosingularitiesannulus} has no singularities in the annulus formed by complex numbers with radius in $[\sqrt{q},1/\sqrt{q}]$.
This implies that in \eqref{eq:equalityofresidues}, the contour $\mathcal C'$ can be deformed to become $\mathcal C$ without encountering any singularities, so that \eqref{eq:equalityofresidues2} holds. 
Using \eqref{eq:conclusioninvolution}, we have proven that the sum of terms in \eqref{eq:sumphitilde} over all diagrams $I$ corresponding to case (2) is zero.

\textbf{Case (3):} For any diagram  $I$ belonging to that case, there exist a number $k<i$ such that the diagram $I$ contains the snippet 
\begin{equation}
	i \rightarrowplus  A \rightarrowplus k \leftarrowminus B \leftarrowplus i+1 ,
	\label{eq:snippet1}
\end{equation} 
	 or 
\begin{equation}
	 	i+1 \rightarrowplus  A \rightarrowplus k \leftarrowminus B \leftarrowplus i,
	 	\label{eq:snippet2}
	 \end{equation} 
 where $A$ and $B$ are (possibly empty) sequences of numbers separated by arrows. Here we do not mean that one line of $I$ has the form of \eqref{eq:snippet1} or \eqref{eq:snippet2}, but rather that one line of the diagram $I$ contains one of those snippets. There may be more arows on the left/right of the numbers $i$ and $i+1$. Since the function $\widetilde\phi$ can be written as 
	 \begin{equation}
	 	  \widetilde \phi(\vec z) = (z_i-z_{i+1})\frac{1-qz_{i}z_{i+1}}{1-z_iz_{i+1}} f\left(z_i; \lbrace z_j\rbrace_{j\neq i,i+1}\right)f\left(z_{i+1}; \lbrace z_j\rbrace_{j\neq i,i+1}\right),
	 	  \label{eq:formofphitilde}
 	  \end{equation}
	 for some function $f$, 
	 the contribution of a diagram $I$ containing \eqref{eq:snippet1} exactly cancels the contribution of the diagram $\sigma(I)$ where $\sigma(I)$ is exactly the same as $I$ up to the replacement of \eqref{eq:snippet1} by \eqref{eq:snippet2} and vice-versa. Again, it is clear that $\sigma$ defined an involution. Due to the prefactor $(z_i-z_{i+1})$ in \eqref{eq:formofphitilde}, we find that 
	$\Res{I} \left\lbrace \widetilde \phi(\vec z) \right\rbrace=-\Res{\sigma(I)} \left\lbrace \widetilde \phi(\vec z) \right\rbrace$, so that we may apply \eqref{eq:conclusioninvolution} to the set of diagrams in case (3).

\textbf{Case (4):} Now we consider the case when the diagram $I$ contains snippets of the form 
\begin{eqnarray}
	A \longleftrightarrow& i &\longleftrightarrow B 	\label{eq:snippet5}\\
	C \longleftrightarrow& i+1&\longleftrightarrow D, \nonumber
\end{eqnarray}
where the arrows $\longleftrightarrow$ may be any type of arrow among  $\rightarrowplus, \leftarrowplus$ and $\leftarrowminus$. 
Consider the diagram $\sigma(I)$ obtained from $I$ by replacing the snippet \eqref{eq:snippet5} by 
\begin{eqnarray}
	A \longleftrightarrow& i+1 &\longleftrightarrow B 	\label{eq:snippet6}\\
	C \longleftrightarrow& i&\longleftrightarrow D,\nonumber
\end{eqnarray}
assuming that the arrows in \eqref{eq:snippet5} and the corresponding ones in \eqref{eq:snippet6} are of the same type. Let us call $a$ the minimal number in the first line of \eqref{eq:snippet5} and $b$ the minimal number in the second line of \eqref{eq:snippet5}, and let us call $a', b'$ the minimal numbers in the lines of \eqref{eq:snippet6} ($(a,b)$ and $(a',b')$ are not necessarily the same). By the change of variables $(z_a,z_b)\to (z_{b'},z_{a'})$, we have that by \eqref{eq:formofphitilde}, 
\begin{equation} \oint_{\mathcal C} \frac{dz_a}{2\I\pi} \oint_{\mathcal C} \frac{dz_b}{2\I\pi} \Res{I} \left\lbrace \widetilde \phi(\vec z) \right\rbrace = -  \oint_{\mathcal C} \frac{dz_a}{2\I\pi} \oint_{\mathcal C} \frac{dz_b}{2\I\pi} \Res{\sigma(I)} \left\lbrace \widetilde \phi(\vec z) \right\rbrace.
\label{eq:symmetrytwovariables}
\end{equation}
The minus term comes from the factor $z_a-z_b$ in $\widetilde \phi(\vec z)$. 
Again,  we may apply \eqref{eq:conclusioninvolution} to the set of diagrams in case (4).   
\begin{example}
	To illustrate how \eqref{eq:symmetrytwovariables} works, let us consider the minimal example of the case $n=2$, $\lambda=(1,1)$, in which case the only possible diagram is a diagram without any arrows. The function $\widetilde\phi$ already calculated in \eqref{eq:specialphitilde} simplifies to 
	\begin{equation*}
	\widetilde \phi(z_1,z_2) =\frac{q(\p-\q)(z_1-z_2)}{(1-qz_{1})(1-qz_{2})} \frac{1-q z_1z_2}{1-z_1z_2}\frac{F_{x}(z_1)}{z_1}\frac{F_{x}(z_2)}{z_2}, 
	\end{equation*}
so that it satisfies $\widetilde\phi(z_1,z_2)=-\widetilde \phi(z_2,z_1)$. Since variables are integrated along the same contour, the change of variables $(z_1,z_2)\to(z_1,z_2)$ yields 
\begin{equation*}
\oint_{\mathcal C} \frac{dz_1}{2\I\pi} \oint_{\mathcal C} \frac{dz_2}{2\I\pi} \widetilde\phi(z_1,z_2) = - \oint_{\mathcal C} \frac{dz_1}{2\I\pi} \oint_{\mathcal C} \frac{dz_2}{2\I\pi} \widetilde\phi(z_1,z_2),
\end{equation*}
so that the integral must equal $0$. This is a special case of \eqref{eq:symmetrytwovariables} where the diagram $I$ has no arrows.
\end{example}

Finally,  having treated the four cases above,  we have shown that when $x_{i+1}=x_i+1$,  the sum over all the terms in \eqref{eq:sumphitilde}  vanishes, so that Condition (2) is satisfied.

\subsubsection{Condition (3)} We need to check that 	
\begin{equation}(\rho -\rho q-1)v_n(t;1,x_2,\dots) + v_n(t;0,x_2,\dots) = 0.\label{eq:BCtocheck}\end{equation}
	The left-hand-side of \eqref{eq:BCtocheck} can be written as 
	\begin{equation}
	\sum_{\lambda \vdash n} \frac{(-1)^{n-\ell(\lambda)}}{m_1!m_2!\dots}\sum_{I\in S(\lambda)} \oint_{\mathcal C} \frac{dz_{i_{\mu_1}}}{2\I\pi}  \oint_{\mathcal C} \frac{dz_{j_{\mu_2}}}{2\I\pi} \dots \; \Res{I} \left\lbrace \widehat \phi(\vec z)\right\rbrace, 
	\label{eq:sumphitilde2}
\end{equation} 
where the function $\widehat\phi$ is now defined by 
\begin{equation} \widehat \phi(\vec z) = \frac{(1-q)(\rho + (1-\rho)z_1)}{1-qz_1} \phi_{0, x_2, \dots, x_n} (\vec z).
\label{eq:defphitildebis}
\end{equation}
	Since $1$ is the minimum of the set $\lbrace 1,\dots, n\rbrace$, any diagram $I$ in the sum \eqref{eq:sumphitilde2} must contain a line of the form 
	\begin{equation}
		A\rightarrowplus 1\leftarrowminus B,
		\label{eq:linecontaining1}
	\end{equation}
	where $A$ and $B$ are possibly empty snippets. 
	Let us also consider the diagram $\sigma(I) $ obtained from $I$ by changing the line \eqref{eq:linecontaining1} to
		\begin{equation*}
		B'\rightarrowplus 1\leftarrowminus A',
		\label{eq:linecontaining1bis}
	\end{equation*} 
	where $A',B'$ are obtained from $A,B$ by reversing the direction of arrows. Again, $\sigma$ is an involution on the set of all diagrams appearing in \eqref{eq:sumphitilde2}. 
	Using the explicit expression for $\widehat\phi$ in \eqref{eq:defphitildebis}, we notice that 
	\begin{equation*} \widehat\phi(z_1,z_2,\dots, z_n) dz_1 = \widehat\phi\Big(1/(qz_1),z_2,\dots, z_n\Big) d(1/(qz_1)).\end{equation*}
	Hence, by the change of variables $z_1\to 1/(qz_1)$ (the contour $\mathcal C$ remains unchanged under this change of variables, up to the orientation of the contour), we have that 
	\begin{equation*} \oint_{\mathcal C} \frac{dz_1}{2\I\pi} \Res{I} \left\lbrace \widehat \phi(\vec z)\right\rbrace = -\oint_{\mathcal C} \frac{dz_1}{2\I\pi} \Res{\sigma(I)} \left\lbrace \widehat \phi(\vec z)\right\rbrace.\end{equation*}
	Again, using \eqref{eq:conclusioninvolution}, this shows that \eqref{eq:sumphitilde2} equals zero.

\subsubsection{Condition (4)}
The modulus of the integral is smaller than the integral of the modulus. Thus, inspecting the dependence in  $x_1$ in the integral formula \eqref{eq:deffunctionv}, and since we are integrating over a finite contour $\mathcal C$, we may find constants 
	$c,C>0$ such that 
	\begin{equation*} \vert v_k(t; \vec x) \vert \leqslant Ce^{c\vert x_1\vert}  \vert v_{k-1}(t;  x_2, \dots, x_n)\vert.\end{equation*}
By recurrence, it implies that Condition (4) is satisfied. 
	
\subsubsection{Condition (5)} When $t=0$, the function $\phi_{\vec x}(\vec z)$ defined in \eqref{eq:defphi} has no singularity  at $z_i=1$ anymore, for all $1\leq i\leq n$. We will show that in that case, we can evaluate the integrals in \eqref{eq:deffunctionv} by residues. It will turn out that almost all residues will cancel each other, except for the residue at  $z_n=0, z_{n-1}=0, \dots, z_1=0$, arising only for the diagram $I$ without any arrow, in which case the residue equals $1$. To formalize this idea, we will proceed by recurrence. We 
will show that for all $n\geq 2$, 
\begin{equation}
	v_n(0;\vec x) = v_{n-1}(0;x_1, \dots, x_{n-1}).
	\label{eq:recurrencetoshow}
\end{equation}
Hence, by recurrence, it suffices to compute 
\begin{equation*} v_1(0; x) =  \oint_{\mathcal C} \frac{dz_1}{2\I\pi} \frac{1-qz_1^2}{1-z_1}\left( \frac{1-z_1}{1-q z_1} \right)^{x_1}\frac{\rho}{\rho + (1-\rho)z_1} \frac{1}{z_1}.
\end{equation*}
The pole at $z_1=\frac{-\rho}{1-\rho}$ is outside the contours, there is no residue at $z_1=1$ since $x_1\geq 1$, and the residue at $z_1=0$ is readily evaluated to be $1$. This shows that 
for all $n\geq 1$ and for all $\vec x\in \Weyl{n}_{\geq 1}$, 
\begin{equation}
	v_n(0;\vec x) = 1,
\end{equation}
so that Condition (5) is satisfied for the empty initial condition. 
In order to prove \eqref{eq:recurrencetoshow}, let us consider \eqref{eq:deffunctionv} and perform the integrations over the variable $z_n$. The variable $z_n$ only arises in terms corresponding to diagrams $I$ where the number $n$ belongs to a line without arrow. This occurs only for partitions $\lambda$ such that $\lambda_{\ell(\lambda)}=1$. In that case, the function $\phi_{\vec x}(\vec z)$ has simple poles at $z_n=0$, $z_n=qz_i$ for $i<n$, and $z_n=1/z_i$ for $i<n$. 

Let us denote by $\Lambda_n$ the set of partitions of $n$. We may decompose 
\begin{equation*} \Lambda_n=\Lambda_n^{(1)}\sqcup \Lambda_n^{(>1)}, \end{equation*}
where $\Lambda_n^{(1)}=\lbrace \lambda\vdash n: \lambda_{\ell(\lambda)}=1\rbrace $ and $\Lambda_n^{(>1)}=\lbrace \lambda\vdash n: \lambda_{\ell(\lambda)}>1\rbrace$. There is a natural bijection between the set $\Lambda_n^{(1)}$ and the set $\Lambda_{n-1}$ (this bijection consists in removing the last part of the partition $\lambda\in \Lambda_n^{(1)}$). 

For a partition $\lambda\in \Lambda_n$, let us decompose $S(\lambda)$ as 
\begin{equation*}S(\lambda) = S^{\rm linked}(\lambda)\sqcup S^{\rm free}(\lambda)\end{equation*}
where $S^{\rm free}(\lambda)$ is the set of diagrams $I\in S(\lambda)$ such that the number $n$ belongs to a line without arrows. This set is nonempty only when $\lambda\in \Lambda_n^{(1)}$. 

Let us first consider the sum 
\begin{equation} \sum_{\lambda \in \Lambda_n^{(1)}} \frac{(-1)^{n-\ell(\lambda)}}{m_1!m_2!\dots}\sum_{I\in S^{\rm free}(\lambda)} \oint_{\mathcal C} \frac{dz_{i_{\mu_1}}}{2\I\pi}  \oint_{\mathcal C} \frac{dz_{j_{\mu_2}}}{2\I\pi} \dots \; \Res{I} \left\lbrace \oint_{\mathcal C}\frac{dz_n}{2\I\pi} \phi_{\vec x}(\vec z)\right\rbrace.
	\label{eq:initialconditiontoevaluate}
	\end{equation}
Evaluating the integral over $z_n$ by residues, we have 
\begin{multline}
\eqref{eq:initialconditiontoevaluate} =  \underbrace{\sum_{\lambda \in \Lambda_n^{(1)}} \frac{(-1)^{n-\ell(\lambda)}}{m_1!m_2!\dots}\sum_{I\in S^{\rm free}(\lambda)} \oint_{\mathcal C} \frac{dz_{i_{\mu_1}}}{2\I\pi}  \oint_{\mathcal C} \frac{dz_{j_{\mu_2}}}{2\I\pi} \dots  \Res{I} \left\lbrace \Res{z_n=0} \;\phi_{\vec x}(\vec z)\right\rbrace}_{:=R_0}\\
+ \sum_{i=1}^{n-1}  \underbrace{\sum_{\lambda \in \Lambda_n^{(1)}} \frac{(-1)^{n-\ell(\lambda)}}{m_1!m_2!\dots}\sum_{I\in S^{\rm free}(\lambda)} \oint_{\mathcal C} \frac{dz_{i_{\mu_1}}}{2\I\pi}  \oint_{\mathcal C} \frac{dz_{j_{\mu_2}}}{2\I\pi} \dots  \Res{I} \left\lbrace \Res{z_n=qz_i} \;\phi_{\vec x}(\vec z)\right\rbrace}_{:=R_i}\\
+\sum_{i=1}^{n-1} \underbrace{ \sum_{\lambda \in \Lambda_n^{(1)}} \frac{(-1)^{n-\ell(\lambda)}}{m_1!m_2!\dots}\sum_{I\in S^{\rm free}(\lambda)} \oint_{\mathcal C} \frac{dz_{i_{\mu_1}}}{2\I\pi}  \oint_{\mathcal C} \frac{dz_{j_{\mu_2}}}{2\I\pi} \dots  \Res{I} \left\lbrace \Res{z_n=1/z_i} \;\phi_{\vec x}(\vec z)\right\rbrace}_{:=R_{-i}}. 
\label{eq:defRi}
\end{multline}
We have that 
\begin{equation*} \Res{z_n=0} \left\lbrace \phi_{\vec x}(\vec z) \right\rbrace= \phi_{x_1, \dots, x_{n-1}}(z_1, \dots, z_{n-1}). \end{equation*}
Using the bijection between $\Lambda_n^{(1)}$ and $\Lambda_n$, and taking into account that there are $m_1(\lambda)$ ways to place the letter $n$ in a line of length $1$ in $I\in S^{\rm free}(\lambda)$, the term $R_0$ in \eqref{eq:initialconditiontoevaluate} can be rewritten as 
\begin{equation} R_0 = \sum_{\tilde\lambda \in \Lambda_{n-1}} \frac{(-1)^{n-1-\ell(\tilde\lambda)}}{m_1(\tilde\lambda)!m_2(\tilde\lambda)!\dots}\sum_{I\in S(\tilde\lambda)} \oint_{\mathcal C} \frac{dz_{i_{\mu_1}}}{2\I\pi}  \oint_{\mathcal C} \frac{dz_{j_{\mu_2}}}{2\I\pi} \dots \; \Res{\sigma(I)} \left\lbrace \phi_{x_1, \dots, x_{n-1}}(z_1, \dots, z_{n-1})\right\rbrace,
		\label{eq:initialevaluated}
\end{equation}
where the partition $\tilde\lambda$ in \eqref{eq:initialevaluated} correspond to the partition $\lambda$ in \eqref{eq:initialconditiontoevaluate} with the last part removed (so that $\ell(\tilde\lambda)=\ell(\lambda)-1$), and the diagram $\sigma(I)$ is obtained from the diagram $I$ by removing the letter $n$. There are $m_1(\lambda)$ diagrams $I$ corresponding to a diagram $\sigma(I)$, all giving the same contribution, and this is consistent with the fact that $m_1(\lambda)=m_1(\tilde\lambda)+1$ so that  $m_1(\lambda)! = m_1(\lambda) \times m_1(\tilde\lambda)!$. Hence, we have shown that 
\begin{equation*} R_0 = v_{n-1}(0;x_1, \dots, x_{n-1}).\end{equation*}
Now we need to explain why all other residues will cancel out. Recall that the quantity that we have to compute, $v_n(0;\vec x)$,  is not exactly given by \eqref{eq:initialconditiontoevaluate} but contains additional terms corresponding to partitions  $\lambda\in \Lambda_n^{(>1)}$ or $\lambda\in \Lambda_n^{1)}$ and $I\in S^{\rm linked}(\lambda)$.  
When $\lambda\in \Lambda_n^{(>1)}$, $S^{\rm free}(\lambda)=\emptyset$, so that we may write that 
\begin{equation}
	v_n(0;\vec x) -\eqref{eq:initialconditiontoevaluate}  = \sum_{\lambda \in \Lambda_n} \frac{(-1)^{n-\ell(\lambda)}}{m_1!m_2!\dots}\sum_{I\in S^{\rm linked}(\lambda)} \oint_{\mathcal C} \frac{dz_{i_{\mu_1}}}{2\I\pi}  \oint_{\mathcal C} \frac{dz_{j_{\mu_2}}}{2\I\pi} \dots \; \Res{I} \left\lbrace \oint_{\mathcal C}\frac{dz_n}{2\I\pi} \phi_{\vec x}(\vec z)\right\rbrace.
	\label{eq:remainingresidues}
\end{equation} 
We will show that all terms in the right-hand-side of \eqref{eq:remainingresidues} are exactly cancelled by terms in the sum $\sum_{i=1}^n R_i+R_{-i}$. In order to prove that,  we will match all non trivial terms in $\sum_{i=1}^n R_i+R_{-i}$ indexed by $\lambda\in \Lambda_n^{(1)}$ and $I\in S^{\rm free}(\lambda)$ to some term indexed by $\tilde \lambda \in \Lambda_n, \tilde I \in S^{\rm linked}(\lambda)$ in \eqref{eq:remainingresidues}. The matching is quite natural but will require some notation to be formalized. Let us first explain the idea. 
In the terms $R_i$ and $R_{-i}$ in \eqref{eq:defRi}, we are summing over  diagrams $I$ where the number $n$ is alone on its line, i.e. no residue is taken in the variable $z_n$, and we take extra residues at $z_n=qz_i$ or $z_n=1/z_i$. All those residues should match with residues associated to the diagrams in \eqref{eq:remainingresidues} where now the number $n$  is not alone, i.e. a residue is taken at $z_n=qz_i$ or $z_n=1/z_i$ for some $i$. To prove this, we need to argue that terms in \eqref{eq:defRi} and terms in \eqref{eq:remainingresidues} can be matched bijectively.

It will be convenient to write the terms $R_i$ and $R_{-i}$ in \eqref{eq:defRi} as 
\begin{equation*}
	R_i = \sum_{\lambda\in\Lambda_n^{(1)}}\sum_{I\in S^{\rm free}(\lambda)} R_i(\lambda,I), \;\;\;	R_{-i} = \sum_{\lambda\in\Lambda_n^{(1)}}\sum_{I\in S^{\rm free}(\lambda)} R_{-i}(\lambda,I)
\end{equation*}
and to write \eqref{eq:remainingresidues} as 
\begin{equation*}
	\eqref{eq:remainingresidues} =\sum_{\lambda \in \Lambda_n} \sum_{I\in S^{\rm linked}(\lambda)} T(\lambda,I). 
\end{equation*}
Let us first fix some $i$, $\lambda\in\Lambda_n^{(1)}$ and $I\in S^{\rm free}(\lambda)$, and consider the term $R_i(\lambda, I)$, that is the term  indexed by $\lambda$ and $I$ in the sum of residues defining $R_i$ in \eqref{eq:defRi}. If there is an arrow bearing a plus sign leading to $i$ in the diagram $I$, that is, if the diagram $I$ contains a snippet of the form $j\rightarrowplus i$ or $i\leftarrowplus j$ for some $j$, then  $R_i(\lambda, I)=0$, because it involves the residue at $z_n=qz_i$ and $z_j=qz_i$, which vanishes due to the factor $z_j-z_n$ in $\phi_{\vec x}(\vec z)$. Thus, the term $R_i(\lambda,I)$ vanishes, unless the number $i$ appears in $I$ at the beginning or the end of a line. We will denote $B_i$ the set of diagrams such that $i$ appears at the beginning of a line and $E_i$ the set of diagrams such that $i$ appears at the end of a line. We have just shown that 
\begin{equation}
	\sum_{\lambda\in\Lambda_n^{(1)}}\sum_{I\in S^{\rm free}(\lambda), I\not\in B_i, I\not\in E_i} R_i(\lambda,I) = 0.
	\label{eq:firstvanishing}
\end{equation}
Consider now $\lambda\in \Lambda_n^{(1)}$, and $I\in S^{\rm free}(\lambda)\cap(B_i\cup E_i)$, that is a diagram $I$ such that the number $i$ appears at the beginning or the end of a line. We may associate to such $(\lambda,I)$ a new diagram $\tilde I$ (see Example \ref{ex:tildeI} below) obtained from $I$ by removing the line containing the number $n$ (recall that $I\in S^{\rm linked}(\lambda)$, so that there is a line containing the number $n$ and no arrow) and adding the snippet  $n\rightarrowplus i$ or $i\leftarrowplus n$ to the line containing $i$.
 We define $\tilde\lambda$ as the partition of $n$ corresponding to the shape of the diagram $\tilde I$. 
 \begin{example}
Consider the partition $\lambda=(3,2,1)$, and the diagram 
	\begin{eqnarray*}
		&&	5 \rightarrowplus 1\leftarrowminus 3\\
	I=	&&	2\leftarrowminus 4\\
		&&	6.
	\end{eqnarray*}
If $i=2$, the diagram $\tilde I$ is obtained by removing the third line and adding the snippet $6\rightarrowplus 2$ to the beginning of the second line, that is 
	\begin{eqnarray*}
\tilde I=	&&	 5 \rightarrowplus 1\leftarrowminus 3\\
	&&	6 \rightarrowplus 2\leftarrowminus 4.
\end{eqnarray*}
so that $\tilde\lambda=(3,3)$. 
\label{ex:tildeI}
 \end{example}
Notice that in general, $\ell(\lambda)=\ell(\tilde\lambda)+1$, so that comparing \eqref{eq:remainingresidues} and \eqref{eq:defRi}, we have 
\begin{equation}
	R_i(\lambda,I) = - T(\tilde\lambda,\tilde I). 
	\label{eq:equalityterm1}
\end{equation}

Let us now consider the term $R_{-i}(\lambda, I)$, that is the term indexed by $\lambda$ and $I$ in the sum of residues defining $R_{-i}$  in \eqref{eq:defRi}. If the number $i$ in $I$ is already connected to an arrow bearing a minus sign, that is, if $I$ contains $i\leftarrowminus j$ for some $j$, the associated residue will vanish for the same reason as above. Furthermore, if the line containing $i$ includes a snippet of the form $j\leftarrowplus i$ or $i\rightarrowplus j$, the associated residue will vanish due to the factor $(1-qz_nz_i)$ in $\phi_{\vec x}(\vec z)$.  Otherwise, the line containing $i$ must be of the form $A\rightarrowplus i$ for some possibly empty snippet $A$, that is $i$ must occur at the end of a line and is connected to an arrow bearing a plus sign. We denote by  $E_i^+$ the set of such diagrams.    We have just shown that 
\begin{equation}
	\sum_{\lambda\in\Lambda_n^{(1)}}\sum_{I\in S^{\rm free}(\lambda),  I\not\in E_i^+} R_{-i}(\lambda,I) = 0.
	\label{eq:secondvanishing}
\end{equation}
Assuming $\lambda$ and $I$ are such that $i$ occurs at the end of a line and is connected to an arrow bearing a plus sign, we associate to $(\lambda,I)$ new partitions and diagrams $(\tilde\lambda, \tilde I)$ where  the new diagram $\tilde I$ is obtained from $I$ by  removing the line containing the number $n$ alone, and adding the snippet $i\leftarrowminus n$ to the line containing $i$. The partition $\tilde \lambda$ is the shape of $\tilde I$ as above. Again, 
$\ell(\lambda)=\ell(\tilde\lambda)+1$, so that comparing \eqref{eq:remainingresidues} and \eqref{eq:defRi}, we have 
\begin{equation}
	R_{-i}(\lambda,I) = - T(\tilde\lambda,\tilde I). 
	\label{eq:equalityterms2}
\end{equation}
From \eqref{eq:equalityterm1} and \eqref{eq:equalityterms2}, we deduce
\begin{equation}
\sum_{i=1}^{n-1} \left(\sum_{\lambda\in\Lambda_n^{(1)}}\sum_{I\in S^{\rm free}(\lambda)\cap(B_i\cup E_i)} R_i(\lambda,I) +  \sum_{\lambda\in\Lambda_n^{(1)}}\sum_{I\in S^{\rm free}(\lambda)\cap E_i^+} R_{-i}(\lambda,I)  \right)
=   -\sum_{\tilde\lambda \in \Lambda_n} \sum_{\tilde I\in S^{\rm linked}(\tilde\lambda)} T(\tilde\lambda,\tilde I).
\label{eq:matchingterms}
\end{equation}
In \eqref{eq:matchingterms}, we have use implicitly that the map $(\lambda,I)\to (\tilde\lambda, \tilde I)$ that matches terms on both sides is a bijection. To check this, it is enough to exhibit the inverse map. For $\tilde \lambda\in \Lambda_n$ and a diagram $\tilde I\in S^{\rm linked}(\tilde\lambda)$, one needs to look at which number $i<n$ the number $n$ is connected by an arrow. Such a number $i$ must exist because $\tilde I\in S^{\rm linked}(\lambda)$. If $i$ and $n$ are connected in $\tilde I$ by an arrow bearing a plus sign,  the term  $T(\tilde\lambda,\tilde I)$ is matched with $-R_i(\lambda, I)$, where the diagram $I$ is obtained from $\tilde I$ by removing the arrow $n\rightarrowplus i$ or $i\leftarrowplus n$ and placing $n$ on an extra line alone (it is easy to check that $I\in B_i\cup E_i$, i.e. $i$ occurs in the beginning or end of a line in $I$). If $i$ and $n$ are connected in $\tilde I$ by an arrow bearing a minus sign, the term  $T(\tilde\lambda,\tilde I)$ is matched with $-R_{-i}(\lambda, I)$, where the diagram $I$ is obtained from $I$ by removing the arrow $i\leftarrowminus n$ and adding the number $n$ to an extra line with no arrow (again, it can be easily checked that $I\in E_i^+$, i.e. the number $i$ occurs in $I$ at the end of a line and is connected to an arrow bearing a plus sign).

Thus, combining \eqref{eq:firstvanishing},  \eqref{eq:secondvanishing} and \eqref{eq:matchingterms},  we have shown that 
\begin{equation*}-\sum_{i=1}^{n-1}( R_i+R_{-i}) = \sum_{\tilde\lambda \in \Lambda_n} \sum_{\tilde I\in S^{\rm linked}(\tilde\lambda)} T(\tilde\lambda,\tilde I) =  \eqref{eq:remainingresidues},
\end{equation*}
so that 
\begin{equation*}v_n(0;\vec x) = \eqref{eq:initialconditiontoevaluate}+\eqref{eq:remainingresidues} =  R_0=v_{n-1}(0;x_1, \dots, x_{n-1}).\end{equation*}
	Hence, we have proved \eqref{eq:recurrencetoshow} and this concludes the proof of Condition (5). 
%
	
	\subsubsection{Conclusion}
	We have shown that the function $v_n(t;\vec x)$ defined in \eqref{eq:deffunctionv} satisfies all the conditions in Proposition \ref{prop:freeevolutionu}. By Proposition \ref{prop:freeevolutionu}, we conclude that \eqref{eq:mixedmomentsQ} holds for any $\vec x\in \Weyl{n}_{\geq 1}$. When $n=1$, the formula reduces to \eqref{eq:firstmoment}, and when $n=2$, the formula reduces to 
	\begin{equation*}
		\EE\left[q^{N_{x_1}(t)+N_{x_2}(t)}\right] = \oint_{\mathcal C} \frac{dz_1}{2\I\pi}\oint_{\mathcal C} \frac{dz_2}{2\I\pi} \phi_{\vec x}(z_1,z_2) - \oint_{\mathcal C} \frac{dz_1}{2\I\pi} \Res{z_2=qz_1} \left\lbrace \phi_{\vec x}(z_1,z_2) \right\rbrace - \oint_{\mathcal C} \frac{dz_1}{2\I\pi} \Res{z_2=1/z_1} \left\lbrace \phi_{\vec x}(z_1,z_2) \right\rbrace. 
	\end{equation*}
Explicitly evaluating  the residues yields \eqref{eq:explicitsecondmoment}. This concludes the proof of Theorem \ref{theo:momentformula}.

\section{Duality for open ASEP on a segment} 
\label{sec:segment}
The arguments in Section \ref{sec:dualityhalfline} and \ref{sec:ODE} can be adapted to the case of open ASEP on a segment with two boundaries. Although we cannot solve explicitly the system of ODEs in this case, we explain in this section how to obtain a closed system of ODEs for observables of the integrated current in open ASEP. 

\begin{definition} Let $\llbracket a,b\rrbracket$ denote the set of integers from $a$ to $b$ included. 
	We define the open ASEP on a segment as a Markov process on the state-space $\lbrace 0,1\rbrace^{\llbracket 1, \ell-1\rrbracket}\times \Z_{\geq 0}$ describing the dynamics of particles between two reservoirs, according to the jump rates depicted in Figure \ref{fig:openASEP}. 
	Configurations are described by occupation variables $\eta=(\eta_1, \dots, \eta_{\ell-1})\in \lbrace 0,1\rbrace^{\llbracket 1, \ell-1\rrbracket}$ together with  the number $N=N_\ell\in \Z$ of particles that have gone through the system (see below for details). For any $x\in \llbracket 1, \ell-1\rrbracket$, a particle jumps from site $x$ to $x+1$ or  from site $x+1$ to $x$ at the same rates $\p$ and $\q$ as in half-line ASEP (or full-line ASEP). As on the half-line, a particle is created or annihilated at the site $1$ at exponential rates
	\begin{equation*} \ratealpha\ (1-\eta_1) \ \  \text{and}\ \  \rategamma\ \eta_1.\end{equation*}
	Further, a particle is created or annihilated at the site $\ell-1$ at exponential rates
	\begin{equation*} \ratedelta\ (1-\eta_1) \ \  \text{and}\ \  \ratebeta\ \eta_1.\end{equation*}
	All these events are independent.  We define the integrated current at site $x$ by
	\begin{equation*} N_x(\eta, N_\ell) = \sum_{i=x}^{\ell-1} \eta_i+ N_\ell,\end{equation*}
	where $N_\ell(t)$ denotes the number of particles that have been annihilated at site $\ell-1$, minus the number of particles that have been created at site $\ell-1$ (this corresponds to the total number of particles that have traveled through the system). 
	We denote by $\Letasegment$ the generator associated to this Markov process, that is the operator acting on functions $f: \lbrace 0,1\rbrace^{\llbracket 1,\ell-1\rrbracket}\times \Z \to \R$ by 
	\begin{multline}
	\Letasegment f(\eta,N)  =  \ratealpha (1-\eta_1) \left( f(\eta_1^+,N)-f(\eta,N)\right) + \rategamma \eta_1 \left( f(\eta_1^-,N)-f(\eta,N)\right) \\ 
	+\ratedelta (1-\eta_{\ell-1}) \left( f(\eta_{\ell-1}^+,N-1)-f(\eta,N)\right) + \ratebeta \eta_{\ell-1} \left( f(\eta_{\ell-1}^-,N+1)-f(\eta,N)\right) \\
	+ \sum_{x=1}^{\ell-2} \left(  \p\eta_x(1-\eta_{x+1}) +  \q\eta_{x+1}(1-\eta_{x}) \right) \left( f(\eta^{x,x+1}, N)-f(\eta,N)\right), 
	\label{eq:defLetasegment}
	\end{multline}
	where for $i=1$ or $i=\ell-1$,  $\eta_i^+$ (resp. $\eta_i^-$) is obtained from $\eta$ by setting $\eta_i=1$ (resp. $\eta_i=0$).    
\end{definition} 
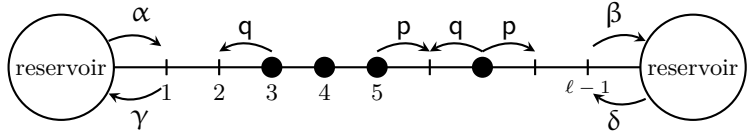
\begin{figure}
	\begin{center} 
		\vspace{-0.3cm}
		\begin{tikzpicture}[xscale=0.7, yscale=0.7, every text node part/.style={align=center}]
		\draw[thick] (0,0) -- (10,0);
		\foreach \x in {1,...,9}
		\draw[thick] (\x,-0.15)--(\x,0.15);
		\foreach \x in {1, ..., 5} {
		\draw (\x, 0.15) -- (\x, -0.15) node[anchor=north]{\footnotesize $\x$};}
		\draw (9, 0.1) -- (9, -0.05) node[anchor=north]{\tiny $\ell-1$};
		\draw[thick] (-1,0) circle(1);
		\draw[thick] (11,0) circle(1);
		\fill (3,0) circle(0.2);
		\fill (4,0) circle(0.2);
		\fill (5,0) circle(0.2);
		\fill (7,0) circle(0.2);
		\draw[-stealth, thick] (3,0.3) to[bend right] (2,0.3);
		\draw[-stealth, thick] (5,0.3) to[bend left] (6,0.3);
		\draw[-stealth, thick] (7,0.3) to[bend right] (6,0.3);
		\draw[-stealth, thick] (7,0.3) to[bend left] (8,0.3);
		\draw (2.5,0.7) node{$\q$};
		\draw (5.5,0.7) node{$\p$};
		\draw (6.5,0.7) node{$\q$};
		\draw (7.5,0.7) node{$\p$};
		\draw (-1,0) node{\footnotesize reservoir};
		\draw (11,0) node{\footnotesize reservoir};
		\draw[-stealth, thick] (-0.1,0.5) to[bend left] (0.9,0.4);
		\draw[stealth-, thick] (-0.1,-0.5) to[bend right] (0.9,-0.4);
		\draw[stealth-, thick] (10.1,0.5) to[bend right] (9.1,0.4);
		\draw[-stealth, thick] (10.1,-0.6) to[bend left] (9.1,-0.5);
		\draw (0.5,1) node{$\ratealpha$};
		\draw (0.5,-1) node{$\rategamma$};
		\draw (9.5,1) node{$\ratebeta$};
		\draw (9.5,-1) node{$\ratedelta$};
		\end{tikzpicture}
		\vspace{-0.3cm}
	\end{center}
	\caption{Jump rates of Open ASEP on $\ell-1$ sites. 
		\label{fig:openASEP}}
\end{figure}
 We now assume that the boundary parameters $\ratealpha, \rategamma, \ratebeta, \ratedelta$ satisfy Liggett's condition \eqref{eq:Liggettcondition} on both sides, that is we assume that 
\begin{equation}
\frac{\ratealpha}{\p}+\frac{\rategamma}{\q} = 1, \;\;\; \frac{\ratebeta}{\p}+\frac{\ratedelta}{\q} = 1,
\label{eq:Liggettcondition2}
\end{equation}
and set $\rho_0=\ratealpha/\p$ and $\rho_\ell=\ratedelta/\q$.
We will use the notation 
\begin{equation*} \Weyl{n}_{\llbracket a, b\rrbracket} :=\left\lbrace \vec x\in \Z^n : a\leq x_1<\dots < x_n\leq b \right\rbrace.\end{equation*}

In order to state a duality, let us define the operator 
$\Lpart^{(n, \rho_0, \rho_\ell)}$ acting on functions $f:\Weyl{n}_{\llbracket 1,\ell\rrbracket}\to \R$ by  
\begin{multline}
	\Lpart^{(n, \rho_0, \rho_\ell)} f(\vec x)  = \sum_{\underset{x_i-x_{i-1}>1}{2\leqslant i\leqslant n} } \p\left( f(\vec x_i^-) -f(\vec x)\right) + \sum_{   \underset{x_{i+1}-x_{i}>1}{1\leqslant i\leqslant n-1}} \q\left( f(\vec x_i^+) -f(\vec x)\right)\\ 
	+\mathds{1}_{x_1>1}\p\left( f(\vec x_1^-) -f(\vec x)\right) -\mathds{1}_{x_1=1} (\p-\q)\rho_0 f(\vec x)\\ 
	+\mathds{1}_{x_n<\ell}\q\left( f(\vec x_n^+) -f(\vec x)\right) +\mathds{1}_{x_n=\ell} (\p-\q)\rho_\ell f(\vec x).
	\label{eq:defLpartsegment}
\end{multline}

 This operator characterizes the time evolution of transition probabilities for a  system of $n$ particles on $\Weyl{n}_{\llbracket 1,\ell\rrbracket}$ performing continuous time simple random walks with jump rates $\p$ and $\q$, under the exclusion constraint, with boundary conditions at $x_1=1$ and $x_n=\ell$ of Robin type (the term ``Robin'' comes from the form of the boundary condition \eqref{eq:boundaryconditionforu} in Proposition \ref{prop:freeevolutionu} below, which is a discrete analogue of the Robin type boundary condition in PDE theory). The same boundary conditions occur in transition probabilities for continuous time random walks with elastic reflections at the boundaries \cite[Sec. 4.1]{corwin2016open}.  
 Hence we will call $\Lpart^{(n, \rho_0, \rho_\ell)}$ the generator of the \emph{mixed Robin boundary $n$-particle ASEP on a segment}, even though it is not the generator of a Markov process.

\begin{theorem} Fix $n\geqslant 1$ and assume that \eqref{eq:Liggettcondition2} holds. 
	The generator \eqref{eq:defLetasegment} of open ASEP on a segment with right and left jump rates $\p$ and $\q$ and the operator $\Lpart^{(n, \rho_0, \rho_\ell)}$, generator of the mixed Robin boundary $n$-particle ASEP on a segment with right and left jump rates $\q$ and $\p$, are dual with respect to 
	\begin{equation}
		H(\eta,N_\ell;\vec x) := \prod_{i=1}^n Q_{x_i}(\eta, N_\ell) \;\;\;\text{ where }\;\;Q_x(\eta, N_\ell)=q^{N_x(\eta, N_\ell)}.
		\label{eq:newdefH}
	\end{equation}
	Equivalently,  for any $(\eta,N_\ell)\in \lbrace 0,1\rbrace^{\llbracket 1,\ell-1\rrbracket}\times \Z$ and $\vec x\in \Weyl{n}_{\llbracket 1,\ell\rrbracket}$, 
	\begin{equation*}
		\Letasegment H(\eta, N_\ell;  \vec x) = \Lpart^{(n, \rho_0, \rho_\ell)} H(\eta, N_\ell; \vec x),
		\label{eq:dualitysegment}
	\end{equation*}
where we recall that $\Letasegment$ acts on functions of $(\eta,N_\ell)$ and $\Lpart^{(n, \rho_0, \rho_\ell)}$ acts on functions of $\vec x$. 
	\label{theo:dualitysegment}
\end{theorem}
\begin{remark}
In the case of the symmetric simple exclusion process on a segment, a similar Markov duality was proved in  \cite[Th. 4.5]{giardina2009duality} (see also \cite[Section 5]{spohn1983long} and \cite{schutz1994non}) with respect to another duality functional. Markov duality for the symmetric simple exclusion process with arbitrary open boundaries is also discussed in  \cite{ohkubo2017dualities}. 

More generally, for different models, Markov dualities between open particle systems with symmetric jump rates and systems with absorbing sites have been previously obtained in the literature \cite{giardina2009duality, giardina2007duality, carinci2013duality, floreani2022orthogonal}. 

In the case of open ASEP, a self-duality is proved in \cite{kuan2021algebraic} for open ASEP with parameters $\alpha,\gamma$ such that $\alpha/\gamma=\p/\q$ and $\beta=\delta=0$. The duality functional is written in algebraic terms and is different from that of Theorem \ref{theo:dualitysegment} (in particular, it depends on boundary parameters through the ratio $\alpha/\gamma$). Another self-duality for open ASEP is also proved in \cite{kuan2022two} for generic $\alpha, \gamma$ but $\beta=\gamma=0$ and yet another duality functional.   
Finally, a reverse Markov duality is proved in  \cite{schutz2022reverse} between the open ASEP on a segment, for parameters $\alpha, \beta, \gamma, \delta$ belonging to a certain manifold, and a closed boundary ASEP on a segment with particle-dependent jump rates. This allows to obtain information on the precise structure of stationary measures with shocks that arise on this specific manifold of parameters. 
\label{rem:dualityopenasep}
\end{remark}

\begin{proof} 
 As in Section \ref{sec:dualityhalfline}, Liggett's condition \eqref{eq:Liggettcondition2} allows to rewrite the generator  $\Letasegment$ as the expectation of a simpler generator on $\ell+1$ sites. 
 More precisely, let us first define the operator 
\begin{multline}
\Lzerozero f(\eta, N)  =  \sum_{x=0}^{\ell-2} \left(  \p\eta_x(1-\eta_{x+1}) +  \q\eta_{x+1}(1-\eta_{x}) \right) \left( f(\eta^{x,x+1}, N)-f(\eta, N)\right), \\ 
+ \p\eta_{\ell-1}(1-\eta_{\ell}) \left( f(\eta^{\ell-1,\ell}, N+1)-f(\eta, N)\right) +  \q\eta_{\ell}(1-\eta_{\ell-1})\left( f(\eta^{\ell-1,\ell}, N-1)-f(\eta, N)\right),
\label{eq:defLeta00}
\end{multline}
acting on functions $f:\lbrace 0,1\rbrace^{\llbracket 0,\ell\rrbracket}\times \mathbb Z_{\geq 0}\to \R$,  that is the generator of half-line ASEP on $\llbracket 0,\ell\rrbracket$  with no injection nor ejection of particles (i.e. with reflecting boundaries), where the integer $N$ keeps track of the number of particles that have traveled between sites $\ell-1$ and $\ell$. Then,  consider a function $f:\lbrace 0,1\rbrace^{\llbracket 1,\ell-1\rrbracket}\times \mathbb Z_{\geq 0}\to \R$ defined on $\ell-1$ occupation variables. It can be also viewed as a function $\lbrace 0,1\rbrace^{\llbracket 0,\ell\rrbracket}\times \mathbb Z_{\geq 0}\to \R$ which does not depend on the variables $\eta_0$ and $\eta_{\ell}$, so that we can apply the operator $\Lzerozero$. Hence,  the operator $\Lzerozero$ maps functions $\lbrace 0,1\rbrace^{\llbracket 1,\ell-1\rrbracket}\times \mathbb Z_{\geq 0} \to  \R$ to functions  $\lbrace 0,1\rbrace^{\llbracket 0,\ell\rrbracket}\times \mathbb Z_{\geq 0}\to \R$. 
 We claim that for $f:\lbrace 0,1\rbrace^{\llbracket 1,\ell-1\rrbracket}\times \mathbb Z_{\geq 0}\to \R$,  
\begin{equation}
\Letasegment f(\eta, N)=E\left[\Lzerozero f (\eta, N)\right], 
\label{eq:Liggetttricksegment}
\end{equation}
where $E$ denotes the expectation with respect to $\eta_0$ and $\eta_\ell$, where $\eta_0=1$ with probability $\rho_0$, and $\eta_\ell=1$ with probability $\rho_\ell$. As in the proof of Theorem \ref{theo:dualityhalf-space}, this is due to the fact that under \eqref{eq:Liggettcondition2}, $E[\p\eta_0]=\ratealpha$, $E[\q(1-\eta_0)]=\rategamma$, $E[\p(1-\eta_\ell)] =\ratebeta$ and $E[\q \eta_\ell]=\ratedelta$.

As we have already noted, the function ${H}(\eta,\vec x)$ was  originally   defined for $\eta\in \lbrace 0,1\rbrace^{\mathbb Z}$ in \eqref{eq:defH}, but when $\vec x\in \Weyl{n}_{\llbracket 1,\ell\rrbracket }$, it depends only on  $(\eta_1, \dots, \eta_{\ell-1}, N_\ell) \in \lbrace 0,1 \rbrace^{\llbracket 1,\ell-1\rrbracket}\times \Z_{\geq 0}$.  Indeed, the value of $H(\eta, \vec x)$ does not depend on the $\eta_i$ for $i\leq 0$ and it depends only on the $\eta_i$ for $i\geq \ell$ through $N_\ell$. 

Moreover, for $\vec x\in \Weyl{n}_{\llbracket 1,\ell\rrbracket }$, when computing $\Letafullspace {H}(\eta,\vec x)$, the summation over all $x\in \mathbb Z$ in \eqref{eq:defgeneratorfullspace} can be restricted to a summation over $x\in \llbracket 0,\ell-1\rrbracket$. 
This implies that for any $\vec x\in \Weyl{n}_{\llbracket 1,\ell\rrbracket }$ and $(\eta, N)\in \lbrace 0,1\rbrace^{\llbracket 1,\ell-1\rrbracket}\times \Z_{\geq 0}$, we have  
\begin{equation}
	\Letafullspace {H}(\tilde\eta,\vec x)=\Lzerozero {H}(\eta,N; \vec x)
	\label{eq:restricted}
\end{equation}
where on the right-hand-side, we recall that $\Lzerozero {H}(\eta,N; \vec x)$ is a function  $\lbrace 0,1\rbrace^{\llbracket 0,\ell\rrbracket}\times \Z_{\geq 0} \to \R$, and on the left-hand-side, $\tilde\eta$ is any  configuration on $\lbrace 0,1\rbrace^{\Z}$ which projects to $(\eta, N)\in \lbrace 0,1\rbrace^{\llbracket 0,\ell\rrbracket}\times \Z_{\geq 0}$ in the sense that  $\tilde\eta_x= \eta_x$ for all $0\leq x\leq \ell$ and the variables $\tilde\eta_y$ for $y>\ell$ are such that $N=\tilde\eta_{\ell}+ N_{\ell+1}(\tilde \eta)$.

Thus, combining \eqref{eq:Liggetttricksegment} and \eqref{eq:restricted}, for all $\vec x\in \Weyl{n}_{\llbracket 1,\ell\rrbracket }$ and $(\eta, N)\in \lbrace 0,1\rbrace^{\llbracket 1,\ell-1\rrbracket}\times \Z_{\geq 0}$, 
\begin{equation}
\Letasegment H(\eta, N; \vec x)=E\left[\Lzerozero H(\eta, N;\vec x)\right]= E\left[\Letafullspace H(\tilde\eta; \vec x)\right],
\label{eq:Liggetttrickapplied2}
\end{equation}
where, again,  $\tilde \eta \in \lbrace 0,1\rbrace^{\mathbb Z}$ is any configuration that projects to  $(\eta, N)\in \lbrace 0,1\rbrace^{\llbracket 0,\ell\rrbracket}\times \Z_{\geq 0}$, and the expectation $E$ simply means that we average over the Bernoulli variables $\eta_0,\eta_\ell$ in the middle, and we average over $\tilde\eta_0, \tilde\eta_\ell$ in the right-hand-side. Now, using Proposition \ref{prop:dualityfull-space}, 
\begin{equation}
\Letasegment H(\eta, N;\vec x)= E\left[\Lpart^{(n)} H(\tilde \eta; \vec x)\right].
\label{eq:Liggetttrickdualitysegment}
\end{equation}
As in the proof of Theorem \ref{theo:dualityhalf-space}, $\Lpart^{(n)} H(\tilde\eta,\vec x)$ is a polynomial in the $Q_0(\tilde\eta), Q_1(\tilde\eta), \dots, Q_{\ell+1}(\tilde\eta)$ which is linear in $Q_0(\tilde\eta)$ and $Q_{\ell+1}(\tilde\eta)$.  Since $Q_0(\tilde\eta)=q^{\tilde\eta_0}Q_1(\tilde\eta)$ and $Q_{\ell+1}(\tilde\eta)=q^{N_{\ell+1}(\tilde\eta)}=q^{N-\tilde\eta_{\ell}}$ the expectation $E$ has the effect of replacing  $Q_0$ by  $(\rho_0 q+1-\rho_0)Q_1$ and replacing $Q_{\ell+1}$ by  $(\rho_\ell/q+1-\rho_\ell)Q_\ell$. 
Since
	\begin{equation*} \p\Big((\rho_0 q+1-\rho_0)-1\Big)=\rho_0(\q-\p),\;\;\; \q\Big((\rho_\ell/q+1-\rho_\ell)-1\Big) = \rho_\ell(\p-\q),
	\end{equation*}
we obtain that $E\left[\Lpart^{(n)} H(\tilde \eta; \vec x)\right] =\Lpart^{(n, \rho_0, \rho_\ell)} H(\tilde \eta; \vec x)$, and since $\tilde\eta$ projects to $(\eta,N)$, the function $\Lpart^{(n, \rho_0, \rho_\ell)} H(\tilde \eta; \vec x)$ can be written as $\Lpart^{(n, \rho_0, \rho_\ell)} H( \eta, N; \vec x)$ when $\vec x\in \Weyl{n}_{\llbracket 1,\ell\rrbracket}$. 
To summarize, we have obtained that for any $(\eta,N)\in \lbrace 0,1\rbrace^{\llbracket 1,\ell-1\rrbracket}\times \Z$ and $\vec x\in \Weyl{n}_{\llbracket 1,\ell\rrbracket}$, 
\begin{equation*}	\Letasegment H(\eta, N;  \vec x) = \Lpart^{(n, \rho_0, \rho_\ell)} H(\eta, N; \vec x),\end{equation*}
which concludes the proof.
\end{proof}
As in Section \ref{sec:dualityhalfline}, the duality from Theorem \ref{theo:dualitysegment} yields a closed system of ODEs characterizing the function $(t,\vec x)\mapsto \mathbb E[\prod_{i=1}^n Q_{x_i}(\eta(t), N_\ell(t))]$.  
\begin{proposition} Assume that \eqref{eq:Liggettcondition2} holds and fix some initial state $(\eta,N)\in \lbrace 0,1\rbrace^{\llbracket 1,\ell-1\rrbracket}\times \Z_{\geq 0}$. There exists a unique function $ u : \R_{+} \times \Weyl{n}_{\llbracket 1,\ell\rrbracket} \to \R$ which satisfies 
	\begin{enumerate}
		\item For all $\vec x \in \Weyl{n}_{\llbracket 1,\ell\rrbracket}$ and $t\in \R_+$, 
		\begin{equation*} \frac{d}{dt} u(t;\vec x) = \Lpart^{(n, \rho_0, \rho_\ell)}  u(t;\vec x); 
			\label{eq:evolutionsegment}
		\end{equation*}
		\item For any $\vec x\in \Weyl{n}_{\llbracket 1,\ell\rrbracket}$, $u(0; \vec x)=H(\eta,N;\vec x)$.
	\end{enumerate}
	The solution is such that for all  $\vec x \in \Weyl{k}_{\llbracket 1,\ell\rrbracket}$ and $t\in \R_+$, $u(t,\vec x) = \EE^{\eta,N}[H(\eta(t),N_\ell(t);\vec x))]$.
	\label{prop:trueevolutionsegment}
\end{proposition}
\begin{proof}
This is a direct consequence of the duality from Theorem \ref{theo:dualitysegment}, as in the proof of Proposition \ref{prop:trueevolution}. Here, the uniqueness of the system of equations is simply due to the fact that the system of ODEs above is finite-dimensional, hence there is no need to impose a condition such as \eqref{eq:boundedgrowth}.
\end{proof}

Again, the duality of operators from Theorem \ref{theo:dualitysegment} can be reformulated as an equality of expectations. 
\begin{corollary}
	Assume that \eqref{eq:Liggettcondition2} holds and fix $\vec x\in \Weyl{n}_{\llbracket 1,\ell\rrbracket}$ and some initial state $(\eta, N)\in \lbrace 0,1\rbrace^{\llbracket 1,\ell-1\rrbracket}\times \mathbb Z_{\geq 0}$. Then, 
	\begin{equation}
		\EE^{\eta, N}\left[H(\eta(t), N_{\ell}(t);  \vec x) \right]  = \mathbb E_{\vec x}\left[H(\eta, N; \vec x(t))\; e^{-(\p-\q)\rho_0\int_0^t \mathds{1}_{x_1(s)=1}ds + (\p-\q)\rho_\ell\int_0^t \mathds{1}_{x_n(s)=\ell}ds} \right],	
		\label{eq:dualityreweighted2}
	\end{equation}
	where $\mathbb E^{\eta, N}$ denotes the expectation with respect to open ASEP on the segment with initial condition $\eta,N$, and $\mathbb E_{\vec x}$ denotes the expectation with respect to the Markov process $\vec x(t)$ on $\Weyl{n}_{\llbracket 1,\ell\rrbracket}$ with generator $\Lpart^{(n, 0,0)}$, that is the generator of  ASEP on $\llbracket 1,\ell\rrbracket$ with closed boundary conditions and jump rates $\p$ to the left and $\q$ to the right. 
	\label{cor:segment}
\end{corollary}
\begin{proof}
Observe that $\Lpart^{(n, \rho_0, \rho_\ell)}=\Lpart^{(n, 0, 0)}+ V^{(\rho_0, \rho_{\ell})}$ where 
$$V^{(\rho_0, \rho_{\ell})}(\vec x) = -(\p-\q)\rho_0 \mathds{1}_{x_1=1} + (\p-\q)\rho_\ell  \mathds{1}_{x_n=\ell}.$$ Since  $V^{(\rho_0, \rho_{\ell})}$ is  a bounded function, the proof of Corollary \ref{cor:segment} is the same as in Corollary \ref{cor:halfspace}. 
\end{proof}
 \begin{remark} Unlike Section \ref{sec:markovdualityhalfspace}, the identity \eqref{eq:dualityreweighted2} cannot be interpreted as an equality of expectations between open ASEP on the segment and some killed process. The exponential factor $e^{-(\p-\q)\rho_0\int_0^t \mathds{1}_{x_1(s)=1}ds}$ in \eqref{eq:dualityreweighted2} could still be interpreted as a killing at rate $\rho(\p-\q)$ when a particle is present at site $1$, but the exponential factor $e^{ (\p-\q)\rho_\ell\int_0^t \mathds{1}_{x_n(s)=\ell}ds}$ cannot. The probabilistic interpretation of such exponential term being less appealing, we have not attempted to rephrase Theorem \ref{theo:dualitysegment} as a Markov duality.  
\end{remark}

As in Section \ref{sec:ODE}, Condition (1) of Proposition \ref{prop:trueevolutionsegment} can be rewritten as Conditions (1) and (2) in Proposition \ref{prop:freeevolutionu} with some boundary conditions when $x_1=1$ or $x_n=\ell$. We obtain the following characterization. 
\begin{proposition} Assume that \eqref{eq:Liggettcondition2} holds and fix some initial state $(\eta,N)\in \lbrace 0,1\rbrace^{\llbracket 1,\ell-1\rrbracket}\times \Z_{\geq 0}$. 
	If the function $ u : \R_{+} \times \llbracket 0,\ell+1\rrbracket^n  \to \R$  solves
	\begin{enumerate}
		\item For all $\vec x \in \llbracket 1,\ell\rrbracket^n$ and $t\in \R_+$, 
		\begin{equation*} \frac{d}{dt} u(t;\vec x) = \Delta^{\p, \q}  u(t;\vec x);
		\end{equation*}
		\item For all $\vec x \in \llbracket 1,\ell-1\rrbracket^n$ such that for some $i\in \lbrace 1, \dots,  n-1\rbrace $,  $x_{i+1}=x_i+1$, we have 
		\begin{equation*} \p u(t; \vec x_{i+1}^-) + \q  u(t; \vec x_i^+) = (\p+\q)  u(t; \vec x);
		\end{equation*}
		\item For all $t\in \R_+$ and $(x_2, \dots, x_n)\in \Weyl{n-1}_{\llbracket 2,\ell\rrbracket}$, 
		\begin{equation}
	  u(t;0,x_2,\dots)=(\rho_0 q +1-\rho_0) u(t;1,x_2,\dots) ;
		\label{eq:boundaryconditionforuleft}
		\end{equation}
		\item For all $t\in \R_+$ and $(x_1, \dots, x_{n-1})\in \Weyl{n-1}_{\llbracket 1,\ell-1\rrbracket}$, 
		\begin{equation}
		  u(t;x_1, \dots, x_{n-1}, \ell+1) = (\rho_\ell/q+1-\rho_\ell) u(t;x_1, \dots, x_{n-1}, \ell);
		\label{eq:boundaryconditionforuright}
		\end{equation}
		\item For any $\vec x\in \Weyl{n}_{\llbracket 1,\ell\rrbracket}$, $u(0; \vec x)=H(\eta,N;\vec x)$;
	\end{enumerate}
	then  for all  $\vec x \in \Weyl{n}_{\llbracket 1,\ell\rrbracket}$ and $t\in \R_+$, $u(t,\vec x) = \EE^{\eta,N}[H(\eta(t),N_\ell(t);\vec x))]$.
	\label{prop:freeevolutionsegment}
\end{proposition}
\begin{proof}As in the proof of Proposition \ref{prop:freeevolutionu}, the conditions (1) (2) (3) and (4) imply Condition (1) of Proposition \ref{prop:trueevolutionsegment}, so that it suffices to apply Proposition \ref{prop:trueevolutionsegment}. 
\end{proof} 
\begin{remark}
The system of ODEs from Proposition \ref{prop:freeevolutionsegment} is a priori amenable for Bethe ansatz, though the hypothetical Bethe ansatz solution would involve the roots of some Bethe equations (needed to enforce the boundary condition at $x=\ell$). It is not presently clear whether there exist simple contour integral formulas for $\EE^{\eta,N}[H(\eta(t),N_\ell(t);\vec x))]$ and we leave this for future consideration. 
\end{remark}
\begin{remark}
	In principle, there should be another approach to solving the systems in Proposition  \ref{prop:freeevolutionsegment} of Proposition \ref{prop:trueevolutionsegment}: directly diagonalize the generator $\Letasegment$ using Bethe ansatz. For ASEP on a ring, Bethe ansatz solutions involve Bethe roots and it is possible to write down some explicit integral formulas \cite{liu2020integral}. On a segment,  eigenfunctions have been studied for instance in \cite{de2005bethe, simon2009construction, crampe2010eigenvectors} and their structure is more complicated. Completeness of the eigenbasis has not been mathematically proven, and due to the complexity of the structure, which increases with $\ell$,  it is not clear how to use them to calculate observables such as $\EE^{\eta,N}[H(\eta(t),N_\ell(t);\vec x))]$ or transition probabilities. 
\end{remark}

\begin{remark}
The open ASEP on a segment has been extensively studied using the Matrix Product Ansatz, a method introduced in \cite{derrida1993exact} to express the stationary distribution of occupation variables $(\eta_1, \dots, \eta_{\ell-1})$. This method allows to compute the expected value of the integrated current for a system starting from stationary initial condition. The method has been refined in a number of directions and some physics references also obtain information on the large deviations of the current (see for instance the thesis \cite{lazarescu2013exact} and references therein). Nevertheless, it does not seem possible yet to compute the exact distribution of $(\eta(t),N_\ell(t))$ using Matrix Product Ansatz, nor the observables $\EE^{\eta,N}[H(\eta(t),N_\ell(t);\vec x))]$.  
\end{remark}

\section{KPZ equation limit}  
\label{sec:KPZ}
In this Section, we provide an application of the formula from Theorem \ref{theo:momentformula}, using the convergence of the half-line ASEP to the KPZ equation in the weakly asymmetric scaling.  The KPZ equation, introduced in \cite{kardar1986dynamic}, is the nonlinear stochastic PDE 
\begin{equation}
\partial_t h(t,x) = \frac{1}{2} \partial_{xx}h(t,x) +\frac 1 2 \left( \partial_x h(t,x)\right)^2+\xi(t,x),
\label{eq:KPZ}
\end{equation}
where $\xi$ is a space-time white noise. For the equation on $\R$, we say that $h(t,x)$,   is a solution if for all $t>0, x\in \R$ ,  $h(t,x)=\log Z(t,x)$ where $Z(t,x)$ is a solution of the multiplicative noise stochastic heat equation (SHE)
\begin{equation}
\partial_t Z(t,x) = \frac{1}{2} \partial_{xx} Z(t,x) +Z(t,x)\xi(t,x),
\label{eq:SHE}
\end{equation}
for which solutions can be defined using stochastic calculus \cite{bertini1995stochastic}. Using Proposition \ref{prop:dualityfull-space} for $n=1$, letting $\mathcal Z(t,x)=q^{N_x(t)}$, one can show that $\mathcal Z(t,x)$ satisfies some discrete approximation of \eqref{eq:SHE}. This was originally observed in \cite{gartner1987convergence}, and then used in \cite{bertini1997stochastic} to prove that, after appropriate scaling (see Section \ref{sec:convASEPKPZ}), $\mathcal Z(t,x)$ weakly converges to a solution of \eqref{eq:SHE} as a function of space and time. 

\subsection{Convergence of half-line ASEP to the KPZ equation on $\R_{>0}$}
\label{sec:convASEPKPZ}
In this section, we are mostly interested in the KPZ equation on $\R_{>0}$. To ensure uniqueness of solutions, one needs to impose a Neumann type boundary condition at $x=0$. As for the KPZ equation on $\R$, we define a solution to the KPZ equation on $\R_{>0}$ through the SHE on $\mathbb R_{>0}$. 
\begin{definition}[\cite{corwin2016open, parekh2017kpz}] We say that $h(t,x)\in C(\mathbb R_{>0}\times C( \mathbb R_{>0}))$ solves the KPZ equation on $\mathbb R_{>0}$ with boundary parameter $A\in \mathbb R$ and narrow wedge initial data if for all $(t,x)\in \mathbb R_{>0}^2$, $h(t,x)=\log Z(t,x)$ where $Z(t,x)$ solves the SHE on $\R_{>0}$ with Robin boundary parameter $A$ and delta initial data, that is 
\begin{equation}
\begin{cases} \partial_t Z(t,x) = \frac{1}{2} \partial_{xx} Z(t,x) +Z(t,x)\xi(t,x),\\
\partial_x Z(t,x) \vert_{x=0} = A Z(t,0) \text{ for all } t>0, \\
\lim_{t\to 0} Z(t,x) =\ratedelta_0(x),
\end{cases}
\label{eq:SHEhalfspace}
\end{equation}
where $\ratedelta_0$ is a Dirac mass at zero. This equation can be made sense of in terms of the heat kernel on $\mathbb R_{>0}$ that satisfies the same boundary condition at $x=0$, we refer to \cite[Definition 2.5]{corwin2016open} and \cite[Proposition 4.3]{parekh2017kpz} for details. 
\label{def:KPZequationhalf-space}
\end{definition}
The convergence of ASEP to the KPZ equation from  \cite{bertini1997stochastic} was extended to half-line ASEP and open ASEP on a segment in \cite{corwin2016open}. The result was restricted there to a certain class of initial data and $A\geq 0$, but the convergence result was extended in \cite{parekh2017kpz} to the narrow wedge initial data and arbitrary $A\in \R$. 
 
Consider half-line ASEP from Definition \ref{def:halflineASEP} with empty initial data and, following  \cite{corwin2016open},  let us  scale 
\begin{equation}
\p=\frac{1}{2}e^{\sqrt{\eps}},\;\;\; \q=\frac{1}{2}e^{-\sqrt{\eps}},\;\;\; \rho=\frac{1}{2}+\sqrt{\eps}\left(\frac{1}{4}+\frac{A}{2} \right).
\label{eq:scalings}
\end{equation} 
We further  define the functions 
\begin{equation} \mathcal Z_t(x) = q^{N_{x+1}(t)-x/2}e^{(\p+\q-1)t}, \;\;\;Z^{\eps}(t,x) = \eps^{-1/2} \mathcal Z_{\eps^{-2}t}(\eps^{-1}x).
\label{eq:defZepsilon}
\end{equation}
Then, \cite[Theorem 1.4]{parekh2017kpz} (see also \cite[Theorem 2.17]{corwin2016open} for another type of initial data) proved that, as $\eps\to 0$, $Z^{\eps}$ weakly converges as a space-time process  to the unique mild solution to the SHE on $\R_{>0}$ with boundary parameter $A$ and delta initial data (Definition \ref{def:KPZequationhalf-space}).

\begin{remark} The boundary condition \eqref{eq:boundaryconditionforu}, when $n=1$, can be rewritten as a discrete Robin type boundary condition 
	$ \rho u(t;1) = \frac{u(t;1) - u(t;0)}{1-q}.$
	One may check that formally plugging the scalings \eqref{eq:scalings}, we obtain the boundary condition $\partial_x \mathbb E\left[ Z(t,x)\right]\vert_{x=0}=A\  \mathbb E[Z(t,0)]$.
\end{remark}

\subsection{Moment formulas} 
In this section, we use Theorem \ref{theo:momentformula} and the convergence recalled in Section \ref{sec:convASEPKPZ}  to prove a formula for the moments of the SHE on $\mathbb R_{>0}$. 
\begin{theorem}
Fix $A>0$, $n\geq 1$, and let $Z(t,x)$ denote the solution of the SHE on $\mathbb R_{>0}$ with Robin boundary parameter $A$ and delta initial data (Definition \ref{def:KPZequationhalf-space}). For any $0\leqslant x_1 \leq \dots\leq x_n $, 
\begin{equation} 
\mathbb E\left[ \prod_{i=1}^n Z(t,x_i) \right] = 2^n\int_{r_1+\I\R} \frac{dw_1}{2\I\pi} \dots \int_{r_n+\I\R} \frac{dw_w}{2\I\pi} \prod_{i<j}\frac{w_i-w_j}{w_i-w_j+1}\frac{w_i+w_j}{w_i+w_j-1} \prod_{i=1}^n e^{\frac{t w_i^2}{2}-x_i w_i}\frac{w_i}{A+w_i}
\label{eq:momentKPZhalfspace}
\end{equation}  
where $0=r_1<r_2-1<\dots <r_n-n+1$.
\label{theo:momentsKPZ} 
\end{theorem}
\begin{remark}The formula \eqref{eq:momentKPZhalfspace} has been anticipated in \cite{borodin2016directed} using the replica method, although it is not clear that the arguments presented in \cite{borodin2016directed} can be turned into a fully rigorous proof.  
	However,  one can approximate the SHE by the partition function of directed polymer models in a half-space. The convergence of directed polymer partition functions to the SHE with Robin boundary was studied in \cite{wu2018intermediate, parekh2019positive, barraquand2022stationaryloggamma}. Exact moment formulas for the log-gamma polymer in a half-quadrant are established in \cite{barraquand2018half}, see in particular \cite[Proposition 7.1]{barraquand2018half}. Taking a scaling limit of that formula should yield another proof of Theorem \ref{theo:momentsKPZ}. 
	\label{rem:otherapproach}
\end{remark}
\begin{remark}
Although the moments of $Z(t,x)$ grow too fast to determine its distribution, the exact probability distribution of $Z(t,0)$ was (non-rigorously) computed from the moments in \cite{borodin2016directed} for $A=0$ and in \cite{Krajenbrink2020} for arbitrary $A$, using conjectural combinatorial simplifications or not fully rigorous methods. In the special case $A=-1/2$ however, the distribution was rigorously computed in \cite{barraquand2018stochastic}, through the stochastic six-vertex model and using a symmetric functions identity that reduces the problem to the asymptotic analysis of Pfaffian Schur measures \cite{borodin2005eynard}, in the spirit of \cite{borodin2016asep}. More recently, in the general case, i.e. for arbitrary $A\in \R$, a Fredholm Pfaffian formula (equivalent to \cite{Krajenbrink2020}) was finally established in \cite{imamura2022solvable} through the log-gamma polymer and   $q$-Whittaker measures,  using a symmetric functions identity from \cite{imamura2021skew} that reduces the problem to the asymptotic analysis of free boundary Schur measures \cite{betea2018free}.
\end{remark}
Before proving Theorem \ref{theo:momentsKPZ}, we establish the following proposition, which is exactly the limit as $\eps\to 0$ of the formula from Theorem \ref{theo:momentformula}, under the scalings \eqref{eq:scalings} and \eqref{eq:defZepsilon}. 
\begin{proposition}
	For any $A>0$ and $0< x_1 <\dots <x_k$,  under the scalings \eqref{eq:scalings}, we have 
	\begin{multline}
		\lim_{\eps\to 0} \mathbb E\left[\prod_{i=1}^n Z^{\eps}(t,x_i) \right] =  2^n\sum_{\lambda \vdash n} \frac{1}{m_1!m_2!\dots}\sum_{I\in S(\lambda)} \int_{\I\R} \frac{dw_{i_{\mu_1}}}{2\I\pi}  \int_{\I\R} \frac{dw_{j_{\mu_2}}}{2\I\pi} \dots  \\ 
		\Res{I} \left\lbrace \prod_{i<j}\frac{w_i-w_j}{w_i-w_j+1}\frac{w_i+w_j}{w_i+w_j-1} \prod_{i=1}^n e^{\frac{t w_i^2}{2}-x_i w_i}\frac{w_i}{A+w_i}\right\rbrace, 
		\label{eq:limitmoments}
	\end{multline} 
	where the vertical line $\I\R$ is oriented from bottom to top and $\Res{I}$ now means that we take residues at $w_j=w_i+1$ whenever the arrows $i\leftarrowplus j$ or $j\rightarrowplus i$ are present in the diagram $I$, and we take a residue at $w_j=-w_i+1$ whenever the arrow $i\leftarrowminus j$ is present in the diagram $I$. 	
	
	Furthermore, for any fixed integer $n$, there exist $\eps_0>0$ and a constant $C=C(n,t)$ such that for any $\eps\in (0,\eps_0)$,  
	\begin{equation}  \left\vert  \mathbb E\left[\prod_{i=1}^n Z^{\eps}(t,x_i) \right]\right\vert \leq C. 
		\label{eq:uniformbound}
	\end{equation}
	\label{prop:limit}
\end{proposition}
\begin{proof}
	Using the definition of $Z^{\eps}(t,x)$ in \eqref{eq:defZepsilon} and the formula from  Theorem \ref{theo:momentformula}, we obtain that for $0< x_1 <\dots <x_k$, 
	\begin{multline} \mathbb E\left[\prod_{i=1}^n Z^{\eps}(t,x_i) \right]=\eps^{-k/2}\left(\frac{\q}{\p} \right)^{\frac{n(n-1)}{2}} \\ \sum_{\lambda \vdash n} \frac{(-1)^{n-\ell(\lambda)}}{m_1!m_2!\dots}\sum_{I\in S(\lambda)} \oint_{\mathcal C} \frac{dz_{i_{\mu_1}}}{2\I\pi}  \oint_{\mathcal C} \frac{dz_{j_{\mu_2}}}{2\I\pi} \dots \; \Res{I} \left\lbrace  \prod_{i<j} \frac{z_i-z_j}{qz_i-z_j} \frac{q^{-1}-z_iz_j}{1-z_iz_j}\prod_{j=1}^n\frac{G_{x_j}(z_j)}{z_j}\right\rbrace,
		\label{eq:momentZepsilon}
	\end{multline} 
	where 
	\begin{equation*} G_x(z) = \frac{\p-\q z^2}{\p-\p z}  \exp\left( \frac{(\sqrt{\p}-\sqrt{\q})^2(\sqrt{\p}+\sqrt{\q} z)^2  \eps^{-2}t}{(1-z)(\p-\q z)}\right) \left( \frac{\sqrt{\p\q}(1-z)}{\p-\q z} \right)^{\eps^{-1}x+1}\frac{\rho}{\rho + (1-\rho)z}.\end{equation*}
	The residues appearing in \eqref{eq:momentZepsilon} all corresponds to simple poles and are easy to compute, so that the formula is explicit. Indeed, for a meromorphic function $f$ of $z_i$ and $z_j$ with a simple pole of order $1$ at $z_j=qz_i$, $\Res{z_j\to qz_i}f(z_i, z_j)$ can be simply computed by replacing $z_j$ by  $qz_i$ in $(z_j-qz_i)f(z_i,z_j)$. Residues of the form $\Res{z_i\to 1/z_j}$ are similarly easily computed. 
	Under the change of variables $z=-1/\sqrt{q}q^{w}$ (recall that $-1/\sqrt{q}q^{w}= -\sqrt{\p/\q}(\q/\p)^{w}=-e^{\sqrt{\eps}(1-2w)}$), 
	\begin{equation*} \frac{G_x(z)dz}{z} = e^{\frac{t w^2}{2}-xw}\frac{wdw}{A+w}+ O(\sqrt{\eps}).\end{equation*}
	Further, under the change of variables $z_i=-1/\sqrt{q}q^{w_i}$, 
	\begin{equation} \frac{z_i-z_j}{qz_i-z_j} \frac{q^{-1}-z_iz_j}{1-z_iz_j}\xrightarrow[\eps\to 0]{} \frac{w_i-w_j}{w_i-w_j+1}\frac{w_i+w_j}{w_i+w_j-1}.
		\label{eq:limitrationalfactors}
	\end{equation}
	In order to ensure that the condition  $\rho>\frac{1}{1+\sqrt{q}}$ in Theorem \ref{theo:momentformula} is satisfied for $\eps$ in a neighborhood of $0$ under the scalings \eqref{eq:scalings}, it is necessary and sufficient to impose $A>0$. Under the change of variables, the contour $ \mathcal C$ become the vertical segment $\I[\frac{-\pi}{2} \eps^{-1/2}, \frac{\pi}{2} \eps^{-1/2}]$.  Taking into account the Jacobian of the change of variables and the orientation of contours, this shows that  taking the pointwise limit of integrands in \eqref{eq:momentZepsilon} leads to \eqref{eq:limitmoments}. 
	In order to prove that the limit holds, we will apply the dominated convergence theorem.
	Let us estimate each of the factors appearing in \eqref{eq:momentZepsilon} under the change of variables $z=-1/\sqrt{q}q^{w}$ considered above.
	We observe that when $z$ varies over a circle around $0$ with radius less than $(1+q^{-1})/2$, the function $z\mapsto \left\vert \frac{\sqrt{\p\q}(1-z)}{\p-\q z} \right\vert$ is maximized when $z$ has minimal real part. In particular, if $z$ varies over a circle of radius less than $1/\sqrt{q}$, 
	\begin{equation}\left\vert \left( \frac{\sqrt{\p\q}(1-z)}{\p-\q z} \right)^{\eps^{-1}x+1}\right\vert 
		\leqslant 1.
		\label{eq:boundx}
	\end{equation}
	The rational factors \eqref{eq:limitrationalfactors} are also easily bounded by a constant so that 
	\begin{equation}
		\left\vert \Res{I} \left\lbrace  \prod_{i<j} \frac{z_i-z_j}{qz_i-z_j} \frac{q^{-1}-z_iz_j}{1-z_iz_j}\right\rbrace \right\vert <C.
		\label{eq:boundRes}
	\end{equation} 
	Then, to estimate the exponential factor, we need to control the real part of the function 
	\begin{equation}  \frac{(\sqrt{\p}+\sqrt{\q} z)^2  \eps^{-1}}{(1-z)(\p-\q z)} = \frac{(1-e^{-2\eps^{1/2}w})^2\eps^{-1}}{(1+e^{-\eps^{1/2}(2w-1)})(1+e^{-\eps^{1/2}(2w+1)})}=:H(w).
		\label{eq:defh}
	\end{equation}
	Under the change of variables that we are considering, the contour $\mathcal C$ for $z$ becomes $\I[\frac{-\pi}{2} \eps^{-1/2}, \frac{\pi}{2} \eps^{-1/2}]$ for $w$. However, since we are taking residues in \eqref{eq:momentZepsilon} we not only need to estimate $\Real[H(w)]$ along the contour, but also expressions of the form $\Real[H(w_1)]+\dots+ \Real[H(w_k)]$, where $\vec w$ is of the form $(w+k-1, w+k-2, \dots, w)$ or of the form  $(w+k-1, \dots, w, 1-w, 2-w, \dots, \ell-w)$ for some integers $k,\ell\geq 0$ such that $k+\ell\leq n$, and $w\in \I\R$. This is the purpose of the following lemma. 
	\begin{lemma}Fix positive integers $k,\ell\leq n$ such that $k+\ell\leq n$. There exist a constant $C$ so that uniformly for $\eps$ in a neighborhood of $0$, for all $y\in [\frac{-\pi}{2} \eps^{-1/2}, \frac{\pi}{2} \eps^{-1/2}]$, 
		\begin{equation}
			\sum_{i=0}^{k-1}\Real[H(\I y+i)]\leq  C (C-y^2)
			\label{eq:lemma}
		\end{equation}
		and for a sequence $(w_1, \dots, w_{k+\ell})$ of the form $(w+k-1, \dots, w, 1-w, 2-w, \dots, \ell-w)$ with $w=\I y$, 
		\begin{equation}
			\sum_{i=1}^{k+\ell}\Real[H(w_i)]\leq  C (C-y^2).
			\label{eq:lemma2}
		\end{equation}
		\label{lem:quadraticdecay}
	\end{lemma}
	\begin{proof}
		Let us consider first the case  $k=1$ for simplicity. The real part of $H[\I y]$ can be computed explicitly as
		\begin{equation*} \Real[H(\I y)] = \frac{-2 \sin(\eps^{1/2}y)^2\eps^{-1}}{\cos(2\sqrt\eps y )+\cosh(\sqrt\eps)},\end{equation*}
		so that, using $\sin(x)\geq x/2$ for $x\in [0, \pi/2]$ and the fact that the denominator can be bounded by a constant,  we obtain that $\Real[H(\I y)] \leq - y^2/2$.  
		
		To prove the more general estimates \eqref{eq:lemma} and \eqref{eq:lemma2}, it is useful to observe that, 
		letting $h(z)=\frac{(1+\sqrt{q}z)^2}{(1-z)(1-q z)}$, we have 
		\begin{equation}
			h(z)-1  = \frac{(1+\sqrt{q})^2}{1-q}\left( \frac{1}{1-z}- \frac{1}{1-qz}\right).
			\label{eq:structureofh}
		\end{equation} 
		Hence, we have the following telescopic sum simplification: 
		\begin{equation*} \sum_{i=0}^{k-1} h(q^iz)-1 = \frac{(1+\sqrt{q})^2}{1-q}\left( \frac{1}{1-z}- \frac{1}{1-q^kz}\right).
		\end{equation*} 
		Using the change of variables $z=-1/\sqrt{q}q^{\I y}$, $q=e^{-2\sqrt\eps}$, we obtain that \eqref{eq:lemma} can be explicitly computed as 
		\begin{equation*} \sum_{i=0}^{k-1}\Real[H(Iy+i)] = k\eps^{-1}+\Real\left[ \eps^{-1}\frac{(1+e^{-\sqrt\eps})^2}{1-e^{-2\sqrt\eps}} \frac{-(1-e^{-2\sqrt\eps k})e^{\sqrt\eps(1-2\I y)}}{(1+e^{\sqrt\eps(1-2\I y)})(1+e^{\sqrt\eps(1-2k-2\I y)})}\right].\end{equation*}
		Using Mathematica, we find that the real part can be computed as 
		\begin{equation*} k\eps^{-1} -\eps^{-1} \frac{\cosh(\sqrt\eps/2)\sinh(\sqrt\eps k)(\cosh(\sqrt\eps k)+\cosh(\sqrt\eps(k-1)\cos(2\sqrt\eps y))}{\sinh(\sqrt\eps/2)(\cos(2\sqrt\eps y)+\cosh(\eps))(\cos(2\sqrt\eps y)+\cosh(\eps(1-2k)))}.\end{equation*}
		This quantity behaves at first order when $\eps\to 0$ as $k\left( \frac{(k-1)(2k-1)}{6}-y^2\right)$, which is consistent with the right-hand-side of \eqref{eq:lemma}. To obtain an estimate uniformly in $\eps$, we  
		use the bounds $e^x\leq 1+x+x^2$, $\sinh(x)\leq x+x^3$ and $1+x^2/2\leq \cosh(x)\leq 1+x^2$ valid on some interval  $(0,c)$ for some fixed constant $c>0$, and the bounds $1-x^2\leq \cos(x)\leq 1-x^2/2+x^4$, valid for all  $x\in (-\pi, \pi)$. After simplifications, we obtain that there exists a constant $C$ so that \eqref{eq:lemma} is satisfied for all $y\in [\frac{-\pi}{2} \eps^{-1/2}, \frac{\pi}{2} \eps^{-1/2}]$, as long as $\eps$ is small enough so that $2\eps^{1/2}k\leq c$ (the condition $2\eps^{1/2}k\leq c$ is necessary to apply the bounds above). This concludes the proof of \eqref{eq:lemma}.

		Now, for a vector $(z_1, \dots, z_{k+\ell})$ of the form $(q^{k-1}z, \dots,qz ,  z, 1/z, q/z, \dots,q^{\ell-1}z),$ we have 
		\begin{equation*} \sum_{i=1}^{k+\ell} h(z_i)-1 = \frac{(1+\sqrt{q})^2}{1-q}\left(1-  \frac{1}{1-q^{\ell}/z}- \frac{1}{1-q^kz}\right).
		\end{equation*} 
		Using the change of variables $z=-1/\sqrt{q}q^{\I y}$, $q=e^{-2\sqrt\eps}$, we obtain that \eqref{eq:lemma2} can be explicitly computed as 
		\begin{multline}
			\sum_{i=0}^{k+\ell}\Real[H(w_i)] = (k+\ell)\eps^{-1}   - \Real\left[\eps^{-1} \frac{(1+e^{-\sqrt\eps})^2}{1-e^{-2\sqrt\eps}} \right. \\ \left. \times \frac{(1-e^{-2\eps^{1/2}(k+\ell)})\left(1+e^{2\eps^{1/2}(k+\ell)}+\left( e^{(2k-1)\eps^{1/2}}+e^{(2\ell+1)\eps^{1/2}}\right) \cos(2\eps^{1/2}y)\right)}{4(\cos(2\eps^{1/2}y)+\cosh(\eps^{1/2}(2k-1)) )(\cos(2\eps^{1/2}y)+\cosh(\eps^{1/2}(2\ell+1)) )}  \right].
		\end{multline}
		Again, using the same bounds as above on the functions $\cos(x)$ and $\cosh(x)$ and $e^x$, we obtain that there exists a constant $C$ so that \eqref{eq:lemma2} is satisfied for all $y\in [\frac{-\pi}{2} \eps^{-1/2}, \frac{\pi}{2} \eps^{-1/2}]$. This concludes the proof of the Lemma \ref{lem:quadraticdecay}. 
	\end{proof}
	At this point, using Lemma \ref{lem:quadraticdecay} and the bounds \eqref{eq:boundx} and \eqref{eq:boundRes} above, the integrand of each term in the sum \eqref{eq:momentZepsilon} can be dominated by a function of the form $Ce^{tC(C-y^2)}$. Hence, by the dominated convergence theorem, we conclude that \eqref{eq:limitmoments} holds. The same bound leads to \eqref{eq:uniformbound}, which concludes the proof of Proposition \ref{prop:limit}. 
\end{proof}
Now we can prove Theorem \ref{theo:momentsKPZ}. 
\begin{proof}[Proof of Theorem \ref{theo:momentsKPZ}]
The structure of the sum over residues appearing in the formula \eqref{eq:limitmoments} in Proposition \ref{prop:limit} is exactly the same as \cite[Section 7.3]{borodin2016directed}, up to the slight difference of conventions explained in Remark \ref{rem:doublecounting}. In particular, using \cite[Eq. (45)]{borodin2016directed} with $p=0$, we obtain that \eqref{eq:limitmoments} can be rewritten as 
\begin{equation}
	\lim_{\eps\to 0} \mathbb E\left[\prod_{i=1}^n Z^{\eps}(t,x_i) \right] =  2^n\int_{r_1+\I\R} \frac{dw_1}{2\I\pi} \dots \int_{r_n+\I\R} \frac{dw_w}{2\I\pi} \prod_{i<j}\frac{w_i-w_j}{w_i-w_j+1}\frac{w_i+w_j}{w_i+w_j-1} \prod_{i=1}^n e^{\frac{t w_i^2}{2}-x_i w_i}\frac{w_i}{A+w_i}
	\label{eq:limitmomentsnested}
\end{equation}
where $0=r_1<r_2-1<\dots <r_n-n+1$, as in the statement of the Theorem. 
We will not recall here the arguments from \cite{borodin2016directed} that lead to the equality between \eqref{eq:limitmoments} and \eqref{eq:limitmomentsnested}. The details can be found in \cite[Section 7.3, Step 1]{borodin2016directed}. Let us simply point out that the purpose of \cite[Section 7.3, Step 1]{borodin2016directed}  was to show that after moving all contours in \eqref{eq:limitmomentsnested} to $\I\R$, one obtains a complicated sums over residues, and that the non-trivial residues are exactly those in the right-hand-side of \eqref{eq:limitmoments}. Let us also stress that the step 2 in \cite[Section 7.3]{borodin2016directed} relies on some unproven conjectural claim, but the arguments in step 1, that we are using here, are complete and do not rely on any unproven claim. 

To finish the proof, we need to relax the strict monotonicity condition on the $x_i$ and prove that the right-hand-side of \eqref{eq:limitmomentsnested} indeed corresponds to the mixed moments of the SHE.  We observe that by definition, $\vert Z_t(x)/Z_t(x+1)\vert \leq \q/\p$ so that for any $0\leqslant x_1 \leq \dots\leq x_n $ 
\begin{equation*} \lim_{\eps\to 0} \mathbb E\left[\prod_{i=1}^n Z^{\eps}(t,x_i) \right] =\lim_{\eps\to 0} \mathbb E\left[\prod_{i=1}^n Z^{\eps}(t,x_i+\eps i) \right],\end{equation*}
which is also given by the expression in the right-hand-side of \eqref{eq:limitmomentsnested}. Hence, we now have a formula for the limit of all mixed moments of $Z^{\eps}(t,x)$. 
Furthermore, the bound \eqref{eq:uniformbound} from Proposition \ref{prop:limit} shows that the sequence of random variables  $\prod_{i=1}^n Z^{\eps}(t,x_i)$ is uniformly integrable, and  since it converges weakly to $\prod_{i=1}^n Z(t,x_i)$ by \cite[Theorem 1.4]{parekh2017kpz}, we deduce that 
\begin{equation*}\lim_{\eps\to 0} \mathbb E\left[\prod_{i=1}^n Z^{\eps}(t,x_i)\right]=\mathbb E\left[\prod_{i=1}^n Z(t,x_i)\right]=\eqref{eq:momentKPZhalfspace}.\end{equation*}
We refer to \cite{ghosal2018moments} for a very similar argument related to the moments of the full-space SHE. This concludes the proof of Theorem \ref{theo:momentsKPZ}. 
\end{proof}
%
\subsection{KPZ equation on $\R_{>0}$ with Dirichlet initial data}
Another case of interest is when  $\rho$ is not scaled close to $1/2$ but fixed, for instance $\rho=1$. This corresponds to sending $A\to +\infty$, so it is natural to expect that one obtains a SHE with Dirichlet boundary condition in the limit. Such boundary condition was considered in \cite{parekh2019positive} in the context of directed polymer models, for Brownian initial data. However, when the half-line ASEP is started from the empty initial condition, it is not entirely clear a priori what should be the initial condition for the KPZ equation in the limit. In the paper \cite{barraquand2022weakly}, to appear,  it will be shown that when $\rho=1$, the correct scaling to consider is 
\begin{equation}
	Z_{Dir}^{\eps}(t,x) = \eps^{-1} \mathcal Z_{\eps^{-2}t}(\eps^{-1}x),
	\label{eq:defZDir}
\end{equation}
where $Z_t(x)$ is still defined as in \eqref{eq:defZepsilon}. In order to rigorously identify the initial data, a sharp estimate on the second moment of $Z_{Dir}^{\eps}(t,x)$ as $\eps \to 0$ is needed. It turns out that usual techniques, based on the fact that $Z_t(x)$ satisfies a discrete variant of the SHE, does not seem to adapt easily to this case where the initial condition is very singular. Theorem \ref{theo:momentformula} however, allows to prove such an estimate. In particular, we obtain the following result which is used in \cite{barraquand2022weakly} to obtain second moment bounds. 
\begin{proposition} Scaling $\p=\frac{1}{2}e^{\sqrt{\eps}}$ and $\q=\frac{1}{2}e^{-\sqrt{\eps}}$ and letting $\rho=1$, we have for any $0\leq x_1\leq \dots \leq x_n$, 
	\begin{equation}
		\lim_{\eps\to 0} \mathbb E\left[\prod_{i=1}^n Z_{Dir}^{\eps}(t,x_i)\right]= 4^n\int_{r_1+\I\R} \frac{dw_1}{2\I\pi} \dots \int_{r_n+\I\R} \frac{dw_w}{2\I\pi} \prod_{i<j}\frac{w_i-w_j}{w_i-w_j+1}\frac{w_i+w_j}{w_i+w_j-1} \prod_{i=1}^n w_i e^{\frac{t w_i^2}{2}-x_i w_i},
		\label{eq:momentKPZDirichlet}
	\end{equation}  
	where $0=r_1<r_2-1<\dots <r_n-n+1$.
\end{proposition}
\begin{proof}
	The proof is the same as the proof of Theorem \ref{theo:momentsKPZ}, except that in the proof of Theorem \ref{theo:momentsKPZ}, when scaling $\rho=1/2+\sqrt{\eps}(1/4+A/2)$, we had that under the change of variables  $z= -e^{\eps^{1/2}(1-2w)}$, 
	\begin{equation*} \frac{\rho dz}{\rho+(1-\rho)z}= \frac{dw}{A+w}+ O(\eps^{1/2}),
	\end{equation*}
	while for $\rho=1$, we have 
	\begin{equation*} \frac{\rho dz}{\rho+(1-\rho)z}= 2\eps^{1/2}+ O(\eps).
	\end{equation*}
	This explains the prefactor $\eps^{-1}$ in \eqref{eq:defZDir} instead of $\eps^{-1/2}$ and the prefactor $4^n$ instead of $2^n$ in front of the integrals in \eqref{eq:momentKPZDirichlet}.
\end{proof}
\begin{remark} 
Moment formulas for the SHE on $\R_{>0}$ with Dirichlet boundary condition, were  obtained in \cite{gueudre2012directed} using the Bethe ansatz. Although the derivation in \cite{gueudre2012directed} is not mathematically rigorous, it seems that the formula  \cite[Eq. (11)]{gueudre2012directed} matches with \eqref{eq:momentKPZDirichlet} modulo the use of \cite[Conjecture 5.2]{borodin2016directed}. The existence and uniqueness of solutions to the SHE with Dirichlet initial data was shown for Brownian type initial data in \cite{parekh2019positive} and the result is extended in the upcoming paper \cite{barraquand2022weakly} to cover the limit of the empty initial data, that is the initial condition considered in \cite{gueudre2012directed}.
\end{remark}

\renewcommand{\emph}[1]{\textit{#1}}

\bibliography{mainbiblio.bib}
\bibliographystyle{goodbibtexstyle} 
\end{document}